\documentclass[11pt]{article}

\usepackage{amsfonts}
\usepackage{amsmath,amssymb}
\usepackage{amscd}
\usepackage{amsthm}
\usepackage{fancyhdr}
\usepackage{color}
\usepackage{hyperref} 
\usepackage{tikz-cd}

\theoremstyle{plain}
\newtheorem{lem}{Lemma}[section]
\newtheorem{thm}[lem]{Theorem}

\newtheorem*{thmnonumber}{Theorem}
\newtheorem{cor}[lem]{Corollary} 
\newtheorem{prop}[lem]{Proposition}

\theoremstyle{definition}
\newtheorem{defi}[lem]{Definition}  
\newtheorem{ex}[lem]{Example} 
 
\theoremstyle{remark}
\newtheorem{rem}[lem]{Remark}

\newcommand{\ZZ}{\ensuremath{\mathbb Z}}
\newcommand{\CC}{\ensuremath{\mathbb C}}
\newcommand{\RR}{\ensuremath{\mathbb R}}
\newcommand{\HH}{\ensuremath{\mathbb H}}

\newcommand{\cF}{\mathcal{F}}

\newcommand{\cS}{\mathcal{S}}

\newcommand{\g}{\mathfrak{g}}

\newcommand{\CX}{\mathfrak{X}}
\newcommand{\CF}{\mathcal{F}}
\newcommand{\CU}{\mathcal{U}}
\newcommand{\CV}{\mathcal{V}}

\newcommand{\CI}{C^\infty}

\newcommand{\CG}{\Gamma}
\newcommand{\fto}{\rightarrow}

\newcommand{\soutar}{\rightrightarrows}

\newcommand{\st}{\hspace{.05in}:\hspace{.05in}}
\newcommand{\y}{\hspace{.04in}\text{and}\hspace{.04in}}

\newcommand{\bb}{\mathbf{b}}
\newcommand{\bt}{\mathbf{t}}                  %target
\newcommand{\bs}{\mathbf{s}}                  %source
\newcommand{\vX}{\mathfrak{X}} %vector fields

\title{{Hausdorff} Morita Equivalence of singular foliations}

\author{Alfonso Garmendia\footnote{KU Leuven, Department of Mathematics, Celestijnenlaan 200B box 2400, BE-3001 Leuven, Belgium.
Email: 
\texttt{alfonso.garmendia@kuleuven.be}} 
\and 
Marco Zambon\footnote{
Email: 
\texttt{marco.zambon@kuleuven.be
}} 
 }

\begin{document}

\date{}

\maketitle
\begin{abstract}
 {  
We introduce a notion of equivalence for singular foliations - understood as suitable families of vector fields - that preserves their transverse geometry. Associated to every singular foliation there is a holonomy groupoid, by the work of Androulidakis-Skandalis. We show that our notion of equivalence is compatible with this assignment, and as a consequence we obtain several invariants. Further, we show that it unifies some of the notions of transverse equivalence for regular foliations that appeared in the 1980's.}
\end{abstract}

\tableofcontents
%%%%%%%%%%%%%%%%%%%%%%
%%%%%%%%%%%%%%%%%%%%%%

\section*{Introduction}
\addcontentsline{toc}{section}{Introduction}
    
 A \emph{regular foliation} is a smooth partition of a manifold into immersed submanifolds of the same dimension, called leaves. The space of leaves is a topological space, and even though it is  typically  not smooth it is possible to apply differential methods to study it. More precisely, 
canonically associated to every regular foliation there is a  smooth (i.e. Lie) groupoid, namely the holonomy groupoid. It can be viewed as a replacement for the space of leaves, since when the latter is {smooth}, the holonomy groupoid is Morita equivalent to the space of leaves (seen as a trivial Lie groupoid).

The transversal geometry of a regular foliation was addressed  in the 1980's and 1990's.
 Haefliger  stated that a property of a regular foliation is \emph{transversal} if it can be described in terms of the Morita equivalence class of its holonomy groupoid (see the first paragraph  of \cite[\S 1.5]{HaefligerHolonomieClassifiants}). 

Molino introduced various notions of \emph{transverse equivalence} of regular foliations compatible with a Riemannian structure (see \cite[\S 2.2 d)]{MolinoOrbit-LikeFol}, where a list of further references is given). His notion of transverse equivalence is given by the requirement that the pullbacks of the foliations to suitable spaces agree, and does not make any reference to the holonomy groupoid.

In this paper we do not restrict ourselves to regular foliations, but study the transverse geometry of  {singular foliations}. However our results have consequences for
regular foliations too: a by-product of this paper is that Haefliger's approach and Molino's approach  
are equivalent, at least for regular foliations whose holonomy groupoid is Hausdorff.\\

\textbf{Statement of results.} We use the term \emph{singular foliation} to refer to a 
suitable choice of submodule of vector fields, as done in \cite{AndrSk}, rather than to a mere smooth partition into immersed submanifolds.
In the same spirit as Molino, we give a geometric definition of   equivalence  of singular foliations. We call it \emph{Hausdorff Morita equivalence}, see Def. \ref{def:defMEfol}.
Examples of Hausdorff Morita equivalent singular foliations can be constructed quite easily,
using the fact that singular foliations can be pushed forward along suitable maps (see \S\ref{subsection:expush}).

We determine many invariants, which should be regarded as constituents of the ``transverse geometry'' of a singular foliation:
 \begin{thmnonumber}[Prop. \ref{prop:corrleaves} and  theorem. \ref{thm:furtherinv}]
If two singular foliations are  Hausdorff Morita equivalent, then:
 \begin{itemize}
\item[a)] the leaf spaces are homeomorphic,
\item[b)] the isotropy Lie groups (and isotropy Lie algebras) of corresponding leaves are isomorphic,
\item[c)] the representations of corresponding isotropy Lie groups on   
normal spaces to the leaves are isomorphic.
\end{itemize}
 \end{thmnonumber} 
Notice that this is in full analogy with Morita equivalence of Lie groupoids, for which 
the space of orbits, the isotropy Lie groups and their normal representations are a complete set of invariants  \cite[theorem.  4.3.1]{MatiasME}.

Several geometric objects have singular foliations naturally associated to them. These
assignments and Morita equivalence   are compatible:
 \begin{thmnonumber}[Corollary \ref{cor:PoisME} and proposition \ref{prop: implications}]
If two source connected Hausdorff Lie groupoids,
two Lie algebroids \cite[\S 6.2]{GinzburgGrot}, or two Poisson manifolds \cite{xuME} are Morita equivalent, then their underlying singular foliations are  Hausdorff Morita equivalent. \end{thmnonumber}	

One feature of singular foliations is they always have a topological groupoid canonically associated to them, called \emph{holonomy groupoid} \cite{AndrSk}. The main result of this paper states that the notion of Hausdorff Morita equivalence is compatible with this assignment:  

 \begin{thmnonumber}[Thm. \ref{thm:MEfolgroids}]
If two singular foliations are Hausdorff Morita equivalent then their holonomy groupoids   are Morita equivalent (as open topological groupoids).
\end{thmnonumber} 
\noindent For regular foliations -- and more generally the projective ones -- the converse statement holds under a Hausdorffness assumption, see proposition \ref{prop:MEholtoMEfol}.
For arbitrary singular foliations {the situation is more involved:} the holonomy groupoid is a topological groupoid, but actually more structure is necessary (a diffeology) in order to recover a singular foliation from its holonomy groupoid. The theory of Morita equivalence for diffeological groupoids has not been developed yet.\\

From the above it is clear that there is a tight connection between the notions of equivalence for singular foliations and for the associated holonomy groupoids. We emphasise that our definition of Hausdorff Morita equivalence is expressed in terms of the singular foliation alone, without making any reference to the associated holonomy groupoid, and as such, it has the advantage of being easy to handle.

{Finally, we remark that} some of our results {relating} Lie groupoids and singular foliations are stated with the assumption that the Lie groupoids   are Hausdorff.
In \S \ref{sec:NHMESingfol} we propose a variation of the notion of {Hausdorff Morita equivalence}  
which should allow to improve these results, by removing the Hausdorff assumption on Lie groupoids. \\

\textbf{Further questions.}
We list a few natural questions  that we hope to address in the future.

\begin{itemize}
\item Morita equivalence is defined among others for  algebras, Lie groupoids, Poisson manifolds.  In all this cases, the category of representations is invariant under Morita equivalence. Does this hold also for singular foliations? It is not even clear what the correct notion of representation is. (A candidate is given by modules for singular foliations regarded as Lie-Rinehart algebras).

\item Do Morita equivalent singular foliations have Morita equivalent $C^*$-algebras?
For the special case of projective foliations (including regular foliations), the answer is positive: we show just after theorem 
\ref{thm:MEfolgroids} that the holonomy groupoids are Morita equivalent as Lie groupoids, and this implies \cite[theorem. 2.4]{Landsman2000} that the associated $C^*$-algebras are.
 
\end{itemize}

\textbf{Conventions and notations.}
{All manifolds are assumed to be Hausdorff, unless otherwise stated. 
However, as is customary, the space of arrows of a Lie groupoid is not assumed to be Hausdorff in general.} 
 
For pullbacks we adopt the following notation. Let $G\soutar M$ be a groupoid, $A\fto M$ a Lie algebroid and $\CF$ a singular foliation on $M$. Given  a smooth  map $\pi\colon P\fto M$, following the notation of \cite{AndrSk} we  denote their pullbacks to $P$ by $\pi^{-1}(G)$, $\pi^{-1}(A)$ and $\pi^{-1}(\CF)$ respectively. {However,} given a vector bundle $E\to M$, the pullback vector bundle is denoted by $\pi^*(E)$. 

{Finally, we use a subscript to denote compact supported objects: for instance, $\Gamma_c(E)$ denotes the compactly supported sections of a vector bundle $E$.}\\

\textbf{Acknowledgments:}   
\noindent  {We are grateful to Dorette Pronk for explanations about topological groupoids, to
George Skandalis for his constructive remarks on this work, and to Kirsten Wang for pointing out reference \cite{MolinoOrbit-LikeFol}. Further we thank Iakovos Androulidakis and Ori Yudilevich for useful discussions. The authors acknowledge partial support by Pesquisador Visitante Especial grant 88881.030367/2013-01 (CAPES/Brazil), and by IAP Dygest, 
the long term structural funding -- Methusalem grant of the Flemish Government,  
  the FWO under EOS project G0H4518N, the FWO research project G083118N (Belgium).
 }

%%%%%%%%%%%%%%%%%%%%%%
%%%%%%%%%%%%%%%%%%%%%% 
\section{Background on singular foliations and pullbacks}\label{sec:ssfol}

{In this short section we recall  singular foliations, as well as properties of  the pullback construction for singular foliations, Lie groupoids and Lie algebroids.} 

\subsection{Singular foliations and their pullbacks}
We   review some notions from the work \cite{AndrSk} by Iakovos Androulidakis and Georges Skandalis.

\begin{defi}\label{def:singfol}
A \textbf{singular foliation} on a manifold $M$ is a $\CI(M)$-submodule $\CF$
{of the compactly supported vector fields}
$\CX_c(M)$, closed under the Lie bracket and locally finitely generated.
A \textbf{foliated manifold} is a manifold with a singular foliation.
\end{defi}

\begin{rem} For any open set $U\subset M$,  
consider the following modules:
\begin{align*}
\iota_U^{-1}\CF&:=\{ X|_U\st X\in \CF \y \text{supp}(X)\subset U\},\\
\widehat{\iota_U^{-1}\CF}&:=\{ X\in \CX(M) \st fX\in \iota_U^{-1}\CF \hspace{.1in} \text{ for all }f\in \CI_c(U)\}.
\end{align*}
We say that $\CF$ is locally finitely generated if for every point of $M$ there is an open neighborhood $U$ and finitely many $X_1,\dots,X_n\in \widehat{\iota_U^{-1}\CF}$ that generate $\iota_U^{-1}\CF$ as a $\CI_c(U)$ module, {in the sense that
any element of $\iota_U^{-1}\CF$ is a   $\CI_c(U)$ linear combination of the $X_i$'s.}
\end{rem}

\begin{rem}
{Let $(M,\CF)$ be a foliated manifold and $x\in M$. Then $F_x:=\{X(x) \st X\in \CF\}$  is a  vector subspace of $T_x M$. Further $\CF_x:= \CF/I_x\CF$ is a finite dimensional vector space (the {fibre} of $\CF$ at $x$), where $I_x$ denotes the ideal of smooth functions on $M$ vanishing at $x$. The evaluation at the point $x$ induces a short exact sequence of vector spaces
$$0\to \g_x^{\CF}\to \CF_x \to F_x\to 0.$$
One checks   that $\g_x^{\CF}$ is a Lie algebra, with Lie bracket induced by the one of vector fields. {It is called {\bf isotropy Lie algebra} at $x$}.}

\end{rem}

Any singular foliation gives rise to a singular distribution that satisfies the assumptions of the  Stefan-Sussmann theorem; therefore, it induces a partition of the manifold into immersed submanifolds called leaves.  
\begin{ex}
{i) Given an involutive regular distribution $D\subset TM$ (by the Frobenius theorem it corresponds to a regular foliation), we obtain a singular foliation $\CF:=\Gamma_c(D)$.}

{
ii) If $N$ is a closed submanifold of a manifold $M$, then $\{X\in \CX_c(M):X|_N\subset TN\}$ is a singular foliation.}
\end{ex}

\begin{defi}\label{def:pullback}
Let $(M,\CF)$ be a foliated manifold  and  $\pi:P\fto M$ a submersion. Consider the following maps of sections:
 \begin{align*}
d\pi\colon  \CX(P) &\fto \CG(P,\pi^*TM), Y\mapsto d\pi Y,\\
\pi^*\colon\CX(M)&\fto \CG(P,\pi^*TM), X\mapsto X\circ \pi.
\end{align*}
{Note that $d\pi^{-1}(\pi^*\CF)$ is the set of projectable vector fields in $P$ which project to elements of $\CF$.}
The {\bf pullback foliation} of $\CF$ under $\pi$ \cite[proposition 1.10]{AndrSk} is the singular foliation on $P$ given by
\[\pi^{-1}(\CF):=\CI_c(P) \cdot d\pi^{-1}(\pi^*\CF).\]
\end{defi}

\begin{rem}
Given an open set $U\subset M$, denoting the inclusion map $\iota_U\colon U\fto M$,   the pullback foliation $\iota_U^{-1}\CF$ is equal to $\{ X|_U\st X\in \CF \y \text{supp}(X)\subset U\}$.
\end{rem}

\begin{prop}
Let $(M,\CF)$ be a foliated manifold, $x\in M$ a point and $X_1,\dots,X_k\in \CF$ vector fields such that their classes in the fibre $\CF_x$ form a basis. There exists a neighbourhood $U\subset M$ of $x$ such that the pullback foliation $\iota^{-1}_U\CF$ is generated by $X_1,\dots,X_k$ as a $\CI_c(U)$ module.
\end{prop}

\begin{rem}\label{rem:sheavesbij}
A singular foliation on a manifold $M$ can be equivalently regarded as an involutive, locally finitely generated subsheaf of $\CX_M$. Here $\CX_M$ denotes
the sheaf of $C^{\infty}$-modules on $M$ given by the smooth vector fields.
 This result can be found in \cite[\S 2]{RoyWang2017}, and we 
 {obtained it independently extending the results of \cite{AZ6}. }
 
 {
More precisely, the bijective correspondence goes as follows. 
Given a singular foliation $\cF$, one obtains a presheaf $S^{\CF}$ declaring that for all open sets $U$ of $M$,   $$\cS^\CF(U):=\widehat{\iota_U^{-1}\CF}.$$
This  presheaf satisfies the gluing axiom, so it is a subsheaf of $\CX_M$. Further $S^\CF(U)$ is involutive for all open sets $U$, and for every point $x\in M$ there exists a neighbourhood $U_x$ such that $S^\CF(U_x)$ is finitely generated. Conversely, given a 
subsheaf $\cS$ of $\CX_M$, one obtains a submodule $\cF:=(\cS(M))_c$ of $\CX_c(M)$, which is a singular foliation when $\cS$ is involutive and locally finitely generated.}
\end{rem}
 
\begin{defi}\label{def:bisub}
Given foliated manifolds {$(M,\CF_M)$ and $(N,\CF_N)$}, a {\bf bisubmersion} between them \cite{AndrSk} consists of manifold $V$ and two submersions 
$\bs\colon V\to M$ and $\bt\colon V\to N$ (not necessarily surjective) such that: \[\bs^{-1}(\CF_M)=\bt^{-1}(\CF_N)=\Gamma_c(ker(d\bs))+\Gamma_c(ker(d\bt)).\]
A {\bf global bisubmersion} for $(M,\CF)$ is a bisubmersion  such that the maps $\bs$ and $\bt$ are surjective.
\end{defi}

\begin{ex}\label{ex:bisub.grpd}
{Let $A$ be a Lie algebroid, with anchor $\#$. Then $\CF:= \#(\CG_c(A))$ is a singular foliation. If $A$ is integrable to a   Hausdorff Lie groupoid   $G\soutar M$  (i.e. $G$ is Hausdorff), then $(G,\bs,\bt)$ is a global bisubmersion for $\CF$.} 
\end{ex}

\subsection{Relation with pullbacks of Lie groupoids and Lie algebroids}

{In example \ref{ex:bisub.grpd} we recalled that Lie algebroids give rise to singular foliations. We now review  pullbacks for Lie algebroids and Lie groupoids (see for instance \cite{MackenzieGrdAld}), and relate them to the notion of pullback for singular foliations. Before we start, we state a lemma that will be used repeatedly, and which is an immediate consequence of \cite[proposition 7.1]{GinzburgGrot}:  }

 \begin{lem}\label{lem:subcon}
	Let $A$ and $B$ be manifolds, { $k\ge 0$}, and $f\colon A\fto B$ a surjective submersion with $k$-connected fibers. If $B$ is $k$-connected then $A$ is $k$-connected.
\end{lem}

\begin{defi}\label{def:pullbackgroid}
{Given a Lie groupoid $G\soutar M$ and a surjective submersion $\pi:P\fto M$, the manifold $$\pi^{-1} G:= P{}_\pi\!\times_{\bt} G {}_{\bs}\!\times_\pi P$$ is the space of arrows of a  Lie groupoid over $P$. (The   source and target maps are the first and third projections  {and the multiplication is induced by the multiplication in $G$}).} This Lie groupoid is called the {\bf pullback groupoid} of $G$ by $\pi$.
\end{defi}

\begin{defi}\label{def:pullbackLA}
Given a Lie algebroid $A$ over a manifold $M$ with anchor $\#:A\fto TM$, and a surjective submersion $\pi\colon P\fto M$, one checks that 
$$\pi^{-1} A:= \pi^*(A) {}_\# \!\times _{d\pi} TP$$ is the total space of a vector bundle over $P$. It has a  natural Lie algebroid structure, with anchor $\hat{\#}:=\text{pr}_2\colon \pi^{-1} A \fto TP$ being the second projection. {The Lie bracket is determined by its {restriction to} ``pullback sections'', which is given by the Lie brackets in $\CX(P)$ and $\CG(A)$}.  We call this Lie algebroid the {\bf pullback algebroid} of $A$ over $\pi$. 
\end{defi}

These two definitions are nicely related by the following lemma:

\begin{lem}\label{GtoA}   
Consider a surjective submersion $\pi:P\fto M$.
\begin{enumerate}
\item[(i)] Let $G$ be a Lie groupoid over $M$, denote by $A$ its Lie algebroid. The Lie algebroid of the Lie groupoid $\pi^{-1} G$ is $\pi^{-1} A$.
\item[(ii)] Let $A$ be an integrable Lie algebroid over $M$, denote by $G$ the source simply connected Lie groupoid integrating it. If the map $\pi$ has simply connected fibers, then the source simply connected Lie groupoid integrating $\pi^{-1} A$ is $\pi^{-1} G$.
\end{enumerate}
\end{lem}
\begin{proof}
The proof of part (i)   can be found in \cite[\S 4.3]{MackenzieGrdAld}, so we address only the proof of part (ii). The Lie groupoid $\pi^{-1} G$ integrates $\pi^{-1} A$ by part (i). Therefore we need to only show that $\pi^{-1} G$ is source simply connected. Take $p\in P$. Its source fiber is
\[\bs^{-1}(p)=\{(q,g,p) \st \pi(p)=\bs(g) \y \pi(q)=\bt (g)\}\simeq P {}_\pi \!\times_\bt \bs^{-1}(\pi(p)).\]
The canonical submersion $\bs^{-1}(p)\fto \bs^{-1}(\pi(p))$ has simply connected fibers, {since the $\pi$-fibers are simply connected.} Using that $\bs^{-1}(\pi(p))$ is simply connected { and lemma \ref{lem:subcon}} we conclude that $\bs^{-1}(p)$ is simply connected.
\end{proof}

The following lemma relates the two definitions above with the pullback of foliations:

\begin{lem}  \label{lem:pullbackalgfol}
Consider a surjective submersion $\pi:P\fto M$. Let $A$ be a Lie algebroid over $M$ with anchor $\#:A\fto TM$. Then the foliation $\CF_P:=\#(\CG_c(\pi^{-1} A))$   equals $\pi^{-1}(\CF_M)$, where $\CF_M:=\#\CG_c(A)$. 
\end{lem}
\begin{proof}
For the inclusion ``$\supset$'' we argue as follows. For all $X \in\pi^{-1}(\cF_M)$ we have:
\begin{equation} \label{MEGrFl}
d\pi(X)=f_1 \pi^*(Y_1)+\dots f_n \pi^*(Y_n)
\end{equation}
for some $Y_1\dots Y_n\in \CF_M$ and $f_1,\dots,f_n\in \CI_c(P)$. There exists $\alpha_1,\dots,\alpha_n\in \Gamma_c(A)$ such that $\#(\alpha_i)=Y_i$. Denote $\hat{\beta}:=f_1 \pi^* \alpha_1+\dots+f_n\pi^* \alpha_n$, {a section of the pullback vector bundle $\pi^*A$}. Using eq.  \eqref{MEGrFl} we get that $(\hat{\beta},X)\in \CG_c(\pi^{-1} A)$ and moreover $\#(\hat{\beta})=X$, so $X\in \#(\CG_c(\pi^{-1} A))$.

For the other inclusion take $(\hat{\beta},X)\in \CG_c(\pi^{-1} A)$. {The module of sections of $A$ is finitely generated due the} Serre–-Swan theorem, hence $\hat{\beta}\in \CG_c(\pi^{*}(A))$ can be written as $\hat{\beta}= f_1 \pi^* \alpha_1+\dots+f_n\pi^* \alpha_n$ for some $f_1,\dots,f_n\in \CI_c(P)$ and $\alpha_1,\dots,\alpha_n\in \Gamma_c(A)$.
{Since $\pi$ is a submersion, each $\# \alpha_i\in \cF_M$ can be lifted via $\pi$ to a vector field $X_i$ on $P$. By construction  $\sum_i f_iX_i$ lies  in $\pi^{-1}(\CF_M)$. The conclusion follows since
the difference $X- \sum_i f_iX_i$ lies in $\Gamma_c(\ker \pi_*)$, and therefore in $\pi^{-1}(\CF_M)$. }
\end{proof}

\section{ {Hausdorff} Morita equivalence of singular foliations}\label{sec:ME} 

 {We introduce {Hausdorff} Morita equivalence for singular foliations (\S \ref{sec:ME1}), display some easy invariants (\S \ref{subsec:firstinv}), and present several classes of examples: elementary ones in \S \ref{subsec:elemex}, other ones obtained pushing forward foliations (typically to quotients by Lie group actions) in \S \ref{subsection:expush}, and more examples in relation to Morita equivalence of Poisson manifolds, Lie groupoids and Lie algebroids in \S \ref{subsec:com}.}

 \subsection{Definition of {Hausdorff} Morita equivalence}\label{sec:ME1}

The following definition is inspired by Ginzburg's definition of weak Morita equivalence of Lie algebroids  \cite[\S 6.2]{GinzburgGrot}, and is along the lines of the definition for   regular foliations given in  \cite[\S 9.2]{CFPois} .
It is a variation of the various notions of ``transverse equivalence'' of regular foliations that appeared in the work of Molino, see \cite[\S 2.2 d)]{MolinoOrbit-LikeFol}.

\begin{defi}\label{def:defMEfol}
Two singular foliations  $(M,\cF_M)$ and $(N,\cF_N)$
are {\bf Hausdorff
Morita equivalent} if there exists a manifold $P$ and two \emph{surjective submersions with connected fibers} $\pi_M:P\fto M$ and $\pi_N:P\fto N$  such that $\pi_M^{-1}\CF_M=\pi_N^{-1}\CF_N$. In this case we write $(M,\cF_M)\simeq_{ME} (N,\cF_N)$.
\[\begin{tikzcd}
  & P \arrow[dl,swap, "\pi_M"] \arrow[dr, "\pi_N"]   &\\
(M,\CF_M)& &(N,\CF_N)   
\end{tikzcd}\]
\end{defi}

{We will show later (see in particular Remark \ref{rem:easyME}) that it is quite easy  to construct   Hausdorff Morita equivalences.}
Later on, in \S \ref{sec:secondlook} and {\S \ref{sec:NHMESingfol}}, we will comment on why it is not desirable to weaken any of the requirements of Def. \ref{def:defMEfol} (except perhaps for $P$ being Hausdorff). 

\begin{lem}\label{lem:HausMEfolER}
Hausdorff Morita equivalence is an equivalence relation on foliated manifolds.
\end{lem}
\begin{proof}
A foliated manifold $(M,\cF_M)$ is equivalent to itself, by means of $(M, Id_M, Id_M)$, therefore this relation is reflexive. It is  clearly symmetric. We now prove that it is transitive.

For the transitivity, let $(M,\CF_M)$, $(N,\CF_N)$ and $(S,\CF_S)$ be foliated manifold such that $M\simeq_{ME} N$ and $N\simeq_{ME} S$. There exists manifolds $P_1,P_2$  and surjective submersions with connected fibers $\pi_M, \pi_N^1, \pi_N^2, \pi_S$ as in this diagram, inducing the above {Hausdorff } Morita equivalences.

\[\begin{tikzcd}
 & & P_1\times_N P_2 \arrow[dl,swap, "\text{Pr}_1"] \arrow[dr, "\text{Pr}_2"] & &\\
& P_1 \arrow[dl, swap, "\pi_M"] \arrow[dr,"\pi_N^1"]& & P_2 \arrow[dl, swap, "\pi_N^2"] \arrow[dr,"\pi_S"] &\\
(M,\CF_M)& &(N,\CF_N) & & (S,\CF_S)
\end{tikzcd}\]
Take $P:= P_1\times_N P_2$, and denote the projections onto the factors by $\text{Pr}_1$  and $\text{Pr}_2$. 
{The commutativity of the diagram implies that} 
\[(\pi_M\circ \text{Pr}_1)^{-1}(\CF_M)=\text{Pr}_1^{-1}((\pi_N^1)^{-1}(\CF_N))={\text{Pr}_2^{-1}((\pi_N^2)^{-1}(\CF_N))}=(\pi_S\circ \text{Pr}_2)^{-1}(\CF_S).\]
{The maps $\pi_M\circ \text{Pr}_1:P\fto M$ and $\pi_S\circ \text{Pr}_2:P\fto S$ are clearly surjective submersions. We    prove below that they
 have connected fibres, allowing to conclude that  $M\simeq_{ME} S$ via $P=P_1\times_N P_2$ and thus finishing the proof.} 
 
{For any $m\in M$ we now show that $(\pi_M\circ \text{Pr}_1)^{-1}(m)$ is connected. Notice that the map}
\[\text{Pr}_1:(\pi_M\circ \text{Pr}_1)^{-1}(m)\fto \pi_M^{-1}(m),\]
is a surjective submersion with connected fibres, because  {the fibre over} $p_1\in \pi_M^{-1}(m)$ is \[\text{Pr}_1^{-1}(p_1)=\{(p_1,p_2)\st p_2\in (\pi_N^2)^{-1}(\pi_N^1(p_1))\}\cong (\pi_N^2)^{-1}(\pi_N^1(p_1)),\]
which is connected. Then using the connectedness of $\pi_M^{-1}(m)$ {and lemma \ref{lem:subcon}} we get that $(\pi_M\circ {Pr_1})^{-1}(m)$ is connected. The same argument shows that $\pi_S\circ \text{Pr}_2:P\fto S$ also  has connected fibres.
\end{proof}

\begin{rem}\label{lem:localsubm}
Assume that $\dim(M)\ge \dim(N)$. Then, for any $x\in M$, there is a neighbourhood  $W$ of $x$ and a submersion $\Psi\colon W\to N$, such that $$ {\iota_W^{-1}\cF_M}=\Psi^{-1}\cF_N,$$
where $\iota_W\colon W\hookrightarrow M$ is the inclusion.
  To see this, choose any local section $\bb \colon W\to P$ of $\pi_M$ that is transversal to the $\pi_N$ fibres, in the sense that $T(\bb(W))+\ker(\pi_N)_*=TP$ at points of $\bb(W)$. {Then $\Psi:=
  \pi_N\circ \bb$ has the desired property, by the functoriality of the pullback.}
\end{rem}

 \begin{rem}
{ In Definition \ref{def:defMEfol} we do not require the property that $\pi_M^{-1}\CF_M$ and $\pi_N^{-1}\CF_N$ equal $\Gamma_c(\ker(\pi_M)_*)+\Gamma_c(\ker(\pi_N)_*)$. (Including this property would deliver exactly the notion of {global bisubmersion} with connected fibers.)  } 
  The main reason for not including this property is that it is not needed to prove any of the features that we want Hausdorff Morita equivalence to have. 
  
  Another reason is that, given a  singular foliation $(M,\cF)$, there may not exist any global bisubmersion between $(M,\cF)$ and itself. Indeed, assume such a global bisubmersion $(U,\bt,\bs)$ exist.
For every $p\in P$ we have $\cF/I_{\bs(p)}\cF\cong \bs^{-1}\cF/I_p(\bs^{-1}\cF)$.
The dimension  of the latter is $\le 2 (dim(U)-dim(M))$, since the map
$$\Gamma_c(\ker\bs_*)/I_p \Gamma_c(\ker\bs_*)\oplus \Gamma_c(\ker\bt_*)/I_p \Gamma_c(\ker\bt_*)\to \bs^{-1}\cF/I_p(\bs^{-1}\cF),\;\; [X]+[Y]\mapsto [X+Y]$$
is surjective. Combining these two facts we see that, at every $x\in M$, the dimension of $\cF/{I_x}\cF$ is bounded above by $2 (dim(U)-dim(M))$. However there exist singular foliations (on non-compact manifolds) for which this dimension is unbounded. A  concrete example is displayed in \cite[lemma 1.3]{AZ1}.
 \end{rem}
 
\subsection{First invariants}\label{subsec:firstinv}

{Roughly speaking, two singular foliations are Hausdorff Morita equivalent if they have the same leaf space and the ``transverse geometry'' at corresponding leaves is the same. 
Here by ``transverse geometry'' we do not mean only the restriction of a singular foliation to a slice transversal to a leaf, but also the effect that holonomy (a global phenomenon) has on the slice. In this subsection we establish a few invariants. In \S  \ref{sec:hol}  
we will see that the global group-like objects associated to singular foliations (namely, their holonomy groupoids) are also Morita equivalent, giving rise to finer invariants.}

\begin{prop}\label{prop:corrleaves}
Let  $(M,\cF_M)$ and $(N,\cF_N)$ be  singular foliations which are Hausdorff Morita equivalent, by means of $(U,\pi_M,\pi_N)$.   Then 
%\mcomment{I had to change the notation here, can you check that all the $M$ and $N$, and $x$ and $y$, are in the correct place? Like in $\pi_M$ etc}

\begin{enumerate}
\item[(i)] There is a homeomorphism between the {leaf space}
 of $(M,\cF_M)$ and the {leaf space} of $(N,\cF_N)$: it maps the leaf through $x\in M$   to {the leaf of $\cF_N$ containing}    $ {\pi_N({\pi_M}^{-1}(x))}$. 
{It preserves the codimension of leaves and the property of being an embedded leaf.}

\item[(ii)] 
{Let $x\in M$ and $y\in N$ be a points lying in corresponding leaves.}
Choose slices $S_x$ at $x$ and $S_y$ at $y$. Then the foliated manifolds $(S_x,{\iota_{S_x}^{-1}\cF_M})$ and $(S_y,{\iota_{S_y}^{-1}\cF_N})$
are diffeomorphic.
\item[(iii)] Let $x\in M$ and $y\in N$ be points lying in corresponding leaves. Then the isotropy Lie algebras $\g^{\cF_M}_x$ and $\g^{\cF_N}_y$ are isomorphic.
\end{enumerate}
\end{prop}
\begin{proof}
(i) {For every leaf $L_M$ on $M$, the preimage $\pi_M^{-1}(L_M)$ is a leaf of 
$\pi_M^{-1}\CF_M=\pi_N^{-1}\CF_N$. Hence it equals $\pi_N^{-1}(L_N)$ for a unique leaf $L_N$ on $N$, which has the same codimension as $L_M$.} {Since $\pi_M$ and $\pi_N$ are continuous open maps, this assignment is a homeomorphism}. {If $L_M$ is an embedded leaf, then a chart on $M$ adapted to $L_M$ induces a  chart on $U$ adapted to $\pi_M^{-1}(L_M)$, and vice versa.} 

(ii) By Definition \ref{def:defMEfol}, it suffices to work with the submersion $\pi_M\colon U\to M$. Take $u\in U$ and let $S_u$ be a transversal for $\pi_M^{-1}\CF_M$ at $u$. Then $S_x:=\pi_M(S_u)$ is a transversal for $\CF_M$ at $x:=\pi_M(u)$. Counting dimensions, {and shrinking $S_u$ if necessary,} we see that ${\pi_M|_{S_u}}\colon S_u\fto S_x$ is a diffeomorphism. The commutativity of the diagram  

\[\begin{tikzcd}
S_u \arrow[r,"\iota_{S_u}"] \arrow[d,swap, "\pi_M|_{S_u}", "\wr"']& U \arrow[d, "\pi_M"]  \\
   S_x  \arrow[r, "\iota_{S_x}"] & M\\
\end{tikzcd}\]
{implies that the singular foliations ${\iota_{S_x}^{-1}\cF_M}$ and $ {\iota_{S_u}^{-1}\cF_U}$ correspond under the above diffeomorphism, where $\cF_U:=\pi_M^{-1}\CF_M$.}

 (iii) follows from (ii), since the isotropy Lie algebra at a point coincides with the isotropy Lie algebra of the transverse foliation at that point, see \cite[Rem. 2.6]{AZ1}.
\end{proof}

\begin{ex}
a) {Given distinct  integers $k,l>0$, the singular foliation on the real line generated by the vector field $x^k \frac{\partial}{\partial x}$ and the one generated by $x^l \frac{\partial}{\partial x}$ lie in different Morita equivalence classes. This can be seen noticing that there are no {diffeomorphism between neighbourhoods $S^k$ and $S^l$ of the origin
that map $x^k \frac{\partial}{\partial x}|_{S^k}$ to the product of  $x^l \frac{\partial}{\partial x}|_{S^l}$ with a no-where vanishing function}, and then applying proposition \ref{prop:corrleaves} ii).
{Notice however that the underlying partitions into leaves and the isotropy {Lie algebras} are the same.}
 }
 
 b) {Consider the singular foliations on $\RR^2$ given by the linear actions of $GL(2,\RR)$
and  $SL(2,\RR)$. They have the same leaves, namely the origin and its complement. The isotropy Lie algebras at the origin are the Lie algebras of $GL(2,\RR)$
and  $SL(2,\RR)$ respectively, hence by proposition \ref{prop:corrleaves} iii) these two singular foliations are not Hausdorff Morita equivalent.
 }
  \end{ex}

{Recall a singular foliation $(M,\cF)$ is \textbf{projective} if there is a vector bundle $A\to M$ such that $\cF \cong \Gamma_c(A)$ as $C^{\infty}(M)$-modules (see \cite{DebordJDG}, where they are called almost regular foliations). In this case $A$ acquires the structure of an almost injective Lie algebroid, i.e. one for which the anchor map is injective on an open dense set. Further, up to isomorphism covering $Id_M$, there is only one 
almost injective Lie algebroid with underlying singular foliation $\cF$.}  
 
 \begin{prop}\label{prop:MEregproj}
Hausdorff Morita equivalence of singular foliations {preserves the following families of singular foliations:}
\begin{enumerate}
\item [(i)] regular foliations
\item [(ii)] projective foliations
\end{enumerate}
\end{prop}
\begin{proof} {Thanks to remark \ref{lem:localsubm}  it suffices to show that, given a surjective submersion  
$\Psi\colon M\to N$   and   a singular foliation $\CF$ on $N$, the pullback foliation $\Psi^{-1}\cF$ is regular (resp. projective) whenever $\CF$ is. For the regular this is clear, implying (i). For the projective case,} 
let $A$ be the almost injective Lie algebroid associated to $\cF$. The pullback Lie algebroid $\Psi^{-1}A$ is also almost injective, as one checks using Def.  \ref{def:pullbackLA}, and its underlying foliation is $\Psi^{-1}\cF$ by lemma \ref{lem:pullbackalgfol}.\end{proof}

{It is not known whether  {all singular foliations  are locally induced by} a Lie algebroid\footnote{Globally this is not the case \cite[lemma 1.3]{AZ1}.}. We now show that Morita equivalence preserves the property that, around a given point, a singular foliation arises from a Lie algebroid.
\begin{prop}\label{prop:folalgloc}
Let  $(M,\cF_M)$ and $(N,\cF_N)$ be {Hausdorff} Morita equivalent singular foliations, and fix a Hausdorff   Morita equivalence  between them.
Let $x\in M$ and $y\in N$ lie on corresponding leaves. Then $\cF_M$ is induced by a Lie algebroid near $x$ if{f} $\cF_N$ is induced by a Lie algebroid near $y$.
\end{prop}
}

\begin{proof}
{Without loss of generality, assume that $\dim(M)\ge \dim(N)$.} By Remark \ref{lem:localsubm}
there is a surjective submersion $\Psi\colon W\to Z$, defined on a neighbourhood  $W$ of $x$ in $M$, such that $\cF_M|_{W}=\Psi^{-1}(\cF_N|_Z).$ {All along this proof we adopt the notation $\cF_M|_{W}:=\iota_W^{-1}\cF_M$, for $\iota_W$ the inclusion. {We can assume that $\Psi(x)=y$.}}

If $A'$ is a Lie algebroid inducing $\cF_N|_Z$, then the pullback Lie algebroid $\Psi^{-1}(A')$ is a Lie algebroid inducing $\cF_M|_{W}$, by lemma \ref{lem:pullbackalgfol}. Conversely, if $A$ is a Lie algebroid inducing $\cF_M|_{W}$, take any slice $S\subset M$ through $x$ transverse to the fibres of $\Psi$. The map $\Psi$ restricts to an isomorphisms between the foliated manifolds $(S, \cF_M|_S)$ and $(Z,\cF_N|_Z)$, {shrinking $Z$ if necessary.} The former  is induced by the restricted Lie algebroid $\#^{-1}(TS)$
(the Lie subalgebroid of $A$ given as the preimage of $TS$ under the anchor map.)
\end{proof}

\subsection{{Elementary examples}}\label{subsec:elemex}

{The next three subsections are dedicated to  examples 
 of  {Hausdorff} Morita equivalent singular foliations, starting with the elementary ones.}

\begin{ex}[Isomorphic foliations]
Two foliated manifolds $(M,\CF_M)$ and $(N,\CF_N)$ are said to be isomorphic if there exists a diffeomorphism $\phi:M\fto N$ such that $\phi^{-1}\CF_N=\CF_M$. Two isomorphic foliated manifolds are {Hausdorff} Morita equivalent.
\end{ex}

\begin{ex}[Full foliations]
Any two connected manifolds $M$ and $N$ with the full foliations {$\vX_c(M)$ and $\vX_c(N)$} are {Hausdorff} Morita equivalent, using $P=M\times N$ and its projection maps. 
\end{ex}

\begin{ex}[Zero foliations]
Two manifolds $M$ and $N$ with the zero foliations are {Hausdorff} Morita equivalent if and only if they are diffeomorphic.
\end{ex}

 \begin{ex}[Simple foliations]
{A regular foliation $F$ on a manifold $M$ is called simple if the leaf space $M/F$ is a smooth manifold such that the projection map  is a submersion. The foliation $F$ on $M$ and the zero foliation on $M/F$ are Hausdorff Morita equivalent.}
\end{ex}

 Given two foliated manifolds $(M,\CF_M)$ and $(N,\CF_N)$, we define the {\bf product foliation} as the following submodule of $\CX_c(M\times N)$:
\[\CF_M\times \CF_N:=\text{Pr}_M^{-1}\CF_M\cap \text{Pr}_N^{-1} \CF_N,\]
where $\text{Pr}_M\colon M\times N\fto M$ and $\text{Pr}_N\colon M\times N\fto N$ are the projections.
 {The following example is worked out in \cite{AlfonsoThesis}.} 
 \begin{ex}[Product foliations]\label{ex:productfol}
Given two pairs of {Hausdorff} Morita equivalent foliated manifolds $(M,\CF_M)\simeq_{ME} (N,\CF_N)$ and $(M',\CF_{M'})\simeq_{ME} (N',\CF_{N'})$, we have $$(M\times M',\CF_M\times \CF_{M'})\simeq_{ME}(N\times N',\CF_N\times\CF_{N'}).$$ 
\end{ex}

\subsection{{Examples obtained by pushing forward foliations}} 
\label{subsection:expush}

{A bisubmersion $(V,\bt,\bs)$ 
between   foliated manifolds $(M,\CF_M)$ and $(N,\CF_N)$ (see Def. \ref{def:bisub}), when $\bs$ and $\bt$ have connected fibres, is a Morita equivalence 
between $\cF_M|_{\bs(V)}$ and $\cF_N|_{\bt(V)}$. Here we construct Morita equivalences of this kind, starting from simple data for which concrete examples can be found quite easily.}

We start reproducing \cite[lemma 3.2]{AZ1}, about quotients of foliated manifolds. 
 
\begin{lem}\label{prop:submfol} Let $\pi\colon P \to M$ be a {surjective} submersion with connected fibres. Let $\cF$ be a singular foliation on $P$, such that $\Gamma_c(\ker d\pi) \subset \cF$. Then there is a unique singular foliation  $\cF_M$ on $M$
with $\pi^{-1}(\cF_M) = \cF$. 
\end{lem}

{
\begin{rem}\label{rem:easyME}
Given two {surjective} submersions $\bs \colon U \to M$ and $\bt \colon U \to N$  with connected fibres, let $\cF_M$ be a singular foliation on $M$ such that $\bs^{-1}(\cF_M)\supset  \Gamma_c(\ker d\bt)$. Then,  by lemma \ref{prop:submfol}, there is a unique singular foliation    $\cF_N$ on $N$  such that  $\bs^{-1}(\cF_M)={\bt}^{-1}(\cF_N)$. {In particular,} $(M,\CF_M)\simeq_{ME} (N,\CF_N)$. {In other words,} we can ``transport'' the foliation ${\cF_M}$ on $M$ to a Hausdorff Morita equivalent foliation on $N$.
 \end{rem}
} 

\begin{cor}\label{cor:pushdowntwo}
 Given two submersions $\bs \colon U \to M$ and $\bt \colon U \to N$  with connected fibres, assume that 
\begin{equation*}
[\Gamma_c(\ker d\bs), \Gamma_c(\ker d\bt)]\subset \Gamma_c(\ker d\bs) + \Gamma_c(\ker d\bt)=:\cF_U.
\end{equation*}
 Then there are  unique foliations $\cF_M$ and $\cF_N$ on $M$ and $N$ respectively such that  $\bs^{-1}(\cF_M)={\bt}^{-1}(\cF_N) = \cF_U$. In particular,  $(M,\CF_M)\simeq_{ME} (N,\CF_N)$.
   \end{cor}
 \begin{proof}
{Apply lemma \ref{prop:submfol} for the foliation $\cF_U$ twice: to the map $\bs$ and to the map $\bt$.}
\end{proof}

 {
An interesting special case {of Cor. \ref{cor:pushdowntwo}} is when the submersions arise from Lie group actions.

\begin{cor}\label{cor:action}
Consider two \emph{connected} Lie groups $G_1,G_2$ acting\footnote{The actions can be both right actions, both left actions, or one right and one left action.}
 freely and properly on a manifold $P$ with commuting actions. Then the following singular foliations are {Hausdorff} Morita equivalent:
\begin{enumerate}
\item the singular foliation on $P/G_1$  given by the induced $G_2$ action, 
\item the singular foliation on $P/G_2$   given by the induced  $G_1$ action.\end{enumerate}
\end{cor}
\begin{proof} Since the {infinitesimal} generators of the $G_1$-action commute with those of the $G_2$-action, the hypotheses of Cor. \ref{cor:pushdowntwo} are satisfied. It is straightforward that the singular foliation on $P/G_1$   induced by the $G_2$ action pulls back to the singular foliation on $P$   induced by the $G_1\times G_2$ action, {and similarly for the singular foliation on $P/G_2$.}
\end{proof}
Notice that the {Hausdorff} Morita equivalence is realised by $P$ with the projection map.
{When the singular foliation on $P$ given by the $G_1\times G_2$ action is  regular}, the induced foliations on $P/G_1$ and $P/G_2$ are also regular.
}

\subsubsection{Pushing forward foliations on Lie groups}

{Now we specialize {Cor. \ref{cor:action}} even further, taking $P$ to be a Lie group   and $G_1$,$G_2$ to be two connected closed subgroups acting  respectively by left and right multiplication. We  present  {two} examples.
\begin{ex}
 Let ${P}=U(2)$, $G_1=SU(2)$, and let $G_2$ consist of the diagonal matrices in $SU(2)$
 (hence $G_2\cong U(1)$). 
{The quotient of the left action of $SU(2)$ on $U(2)$ is $SU(2)\backslash U(2)\cong S^1$, since the homomorphism $det\colon U(2)\to S^1$ has kernel $SU(2)$.}  The   action of $G_2$ on $U(2)$ by right multiplication descends to the trivial action on $S^1$.
Hence on $S^1$ we obtain the (regular) foliation by points.
By Cor. \ref{cor:action}, it is { Hausdorff} Morita equivalent to the (regular) foliation on $U(2)/G_2$ by orbits of the left $SU(2)$-action.
\end{ex}

\begin{ex}
{We apply Cor. \ref{cor:action} to actions of the Lie groups $SO(2n)$ and $U(n)$ on   $P=SO(2n+1)$.}
We can include $SO(2n)$ in $SO(2n+1)$ {as matrices with  $1$ in the bottom right corner}. {Left multiplication induces a left action of $SO(2n)$ on $SO(2n+1)$ with quotient manifold $SO(2n)\backslash SO(2n+1) \cong S^{2n}$.}

On the other hand, we can include $U(n)$ in $SO(2n+1)$ {as the unitary matrices with  $1$ in the bottom right corner.}
{Right multiplication induces a right action of $U(n)$ on $SO(2n+1)$. The quotient manifold is $$SO(2n+1)/U(n) \cong J(2n+2),$$ where $J(2n+2)$ denotes the set of complex structures in $\RR^{2n+2}$ preserving the canonical inner product and orientation.} {In fact,   there is a diffeomorphism $SO(2n+1)/U(n) \cong SO(2n+2)/U(n+1)$, 
induced by the {transitive} action of $SO(2n+1)$ on $SO(2n+2)/U(n+1)$ inherited from the  left multiplication, which has isotropy group $U(n)$. In turn, $SO(2n+2)/U(n+1)\cong  J(2n+2)$ by considering the 
 action of $SO(2n+2)$ on $J(2n+2)$ by pullbacks, which has isotropy group $U(n+1)$.}

Hence by Corollary. \ref{cor:action} the following singular foliations are { Hausdorff} Morita equivalent: 
\begin{itemize}
\item the singular foliation on $J(2n+2)$  induced by the action of ${SO(2n)\subset SO(2n+2)}$  via pullbacks, 
\item the singular foliation on $S^{2n}$  induced by the  action by {right   matrix multiplication of $U(n)\subset SO(2n+1)$.}
\end{itemize}
{Note that the South pole and North pole of $S^{2n}$ are the only fixed points of the action of $U(n)$, therefore on $S^{2n}$ we have a genuinely singular foliation. 
As a consequence of Morita equivalence, the foliation on $J(2n+2)$ is also non-regular.} 
\end{ex}
 
{In Cor. \ref{cor:action}}, the  case in which $P$ is a  Lie group can be rephrased efficiently as follows.
\begin{prop}\label{prop:MEquottrans}
Fix a Lie group $G$ with two transitive (left) actions on  manifolds $M_1$ and 
 $M_2$. Fix points $p_i\in M_i$, {and assume that the isotropy groups are connected}.
Then the following singular foliations are {Hausdorff} Morita equivalent:
\begin{enumerate}
\item  the singular foliation on $M_1$ induced by the restricted action of the isotropy group $G_{p_2}$,
\item the singular foliation on $M_2$ induced by the restricted  action of the isotropy group $G_{p_1}$.\end{enumerate}
\end{prop}
\begin{proof}
Recall that the transitive left action of $G$ on $M_2$ induces a diffeomorphism $$G/G_{p_2}\cong M_2, \;[g]\mapsto gp.$$ Here $G/G_{p_2}$ is the manifold of left cosets, i.e. the quotient of $G$ by the $G_{p_2}$-action by \emph{right} multiplication. 

Instead of\footnote{Notice that the singular foliation induced by a right action on a manifold agrees with the one induced by the corresponding left action.}
 using the left action of $G$ on $M_1$, for convenience we consider the associated right action, defined by $m\cdot g:=g^{-1}\cdot m$ for all $g\in G$ and $m\in M_1$. Notice that the isotropy groups at $p_1$ for the two actions are the same.
Similarly to the above,  the transitive right action of $G$ on $M_1$ induces a diffeomorphism $G_{p_1} \backslash G\cong M_1$, where  $G_{p_1} \backslash G$ is the quotient of $G$ by the $G_{p_1}$-action   by \emph{left} multiplication.

Now apply Cor. \ref{cor:action} to the manifold $P:=G$, to the action of $G_{p_1}$ by left multiplication and the action of $G_{p_2}$ by right multiplication. This delivers
{Hausdorff} Morita equivalent singular foliations on $G_{p_1} \backslash G$ and $G/G_{p_2}$.
The residual action of $G_{p_1}$ on the quotient $G/G_{p_2}$ is just the restriction\footnote{This follows from the fact that, under the diffeomorphism $G/G_{p_2}\cong M_2$, the left action of $G$ on $M_2$ corresponds to the action on $G/G_{p_2}$ induced by left multiplication on $G$.} of the left $G$-action on $M_2$. 
Similarly, 
the residual action of $G_{p_2}$ on the quotient $G_{p_1} \backslash G$ is just the restriction of the right $G$-action on $M_1$. 
\end{proof}
}
 
\begin{ex}
The Lie group $SL(2,\RR)$ has a natural transitive action on $\RR^2\backslash\{(0,0)\}$ {by left matrix multiplication}. It is easy to check that the isotropy group at $(1,0)$ is:
\[\left\{ \left(\begin{matrix} 1 & t \\ 0 & 1\end{matrix}\right) \st t\in \RR \right\}.\]

The same Lie group also has a transitive action by Moebius transformations on $\HH\subset \CC$, the {open} upper half-plane. It is easy to check that the isotropy group at $i$ is:
\[\left\{ \left(\begin{matrix} \cos(\theta) & \sin(\theta) \\ -\sin(\theta) & \cos(\theta)\end{matrix}\right) \st \theta\in \RR \right\}.\]

Hence by proposition \ref{prop:MEquottrans} the following singular foliations are {Hausdorff} Morita equivalent ({as can be verified easily directly too)}: 
\begin{itemize}
\item the foliation on $\RR^2\backslash\{(0,0)\}$  given by {concentric circles about the origin},
\item the foliation on $\HH$ by {horizontal lines}. 
\end{itemize}
\end{ex}

\begin{ex}\label{ex:Gr}
The Lie group $SO(n)$ has a transitive   action on the sphere $S^{n-1}\subset \RR^n$ by left matrix multiplication. The stabiliser of the South Pole $p\in S^{n-1}$ consist of matrices with $1$ in the lower right corner, and therefore is isomorphic to $SO(n-1)$.

The same Lie group $SO(n)$ also has a transitive  action on the oriented Grassmannian $\widetilde{Gr_{2,n}}$, the space of oriented planes in $\RR^n$ (it has dimension $2(n-2)$). This action is induced by the action on column vectors in $\RR^n$ by matrix multiplication. The stabiliser of the plane 
spanned by the last two canonical basis vectors $e_{n-1}$ and $e_n$ is $G_1:=SO(n-2)\times SO(2)$. 

Hence by proposition 
\ref{prop:MEquottrans} the following singular foliations are {Hausdorff} Morita equivalent: 
\begin{itemize}
\item the singular foliation on $\widetilde{Gr_{2,n}}$  induced by the   $SO(n-1)$ action
\item the singular foliation on $S^{n-1}$  induced by the  $G_1$ action by   matrix multiplication.
\end{itemize}
Notice that these actions have orbits of  codimension $1$, $2$ and $n-2$.  
\end{ex}

\subsection{Examples obtained from Morita equivalence of related objects}
 \label{subsec:com}

{Several geometric objects (including Poisson manifolds, Lie algebroids, Lie groupoids) have underlying singular foliations. We show that two such objects that are  equivalent (in a suitable sense) induce Hausdorff Morita equivalent singular foliations.} 
 
\subsubsection{{Examples from Morita equivalent Poisson manifolds}}

{A Poisson manifold $(M,\Pi_M)$ gives rise to a Lie algebroid (its cotangent bundle) and hence to a singular foliation $\cF_{\Pi_M}$. Explicitly, $\cF_{\Pi_M}$
consists of $C^{\infty}_c(M)$-linear combinations of Hamiltonian vector fields on $M$.
In particular, the leaves of $\cF_{\Pi_M}$ are exactly the symplectic leaves {of the Poisson structure}.
}

\begin{ex}\label{ex:fulldualpair}
Let $(M,\Pi_M)$ and $(N,\Pi_N)$  be Poisson manifolds. A \emph{full dual pair} \cite[\S 8]{We} consists of a symplectic manifold $(U,\omega)$ with surjective submersions $\bs\colon U\to M$ and $\bt\colon U\to N$ which are Poisson and anti-Poisson maps respectively, and such that $ker(d_u\bs)$ and $ker(d_u\bt)$ are symplectic orthogonal subspaces of $T_uU$ for all $u\in U$. 
Notice that $\Gamma(\ker d\bs)$ is generated  by $\{X_{\bt^*g}:g\in C^{\infty}(N)\}$ as a $C^{\infty}(U)$-module, while $\Gamma(\ker d\bt)$ is generated by $\{X_{\bs^*g}:g\in C^{\infty}(M)\}$. Here we denote by $X_F$ the Hamiltonian vector field of the function $F$.

A full dual pair with connected fibres is a global bisubmersion with connected fibres  for the foliations  $\cF_{\Pi_M}$ and  $\cF_{\Pi_N}$ (see Def. \ref{def:bisub}).
Indeed, since $\bs$ is a Poisson map, for any Hamiltonian vector field $X_{g}$ a $\bs$-lift is given by $X_{\bs^*g}$, hence 
$$\bs^{-1}(\cF_{\Pi_M})=Span_
{C^{\infty}_c(U)}(\{X_{\bs^*g}:g\in C^{\infty}(M)\}+\Gamma(\ker d\bs))=\Gamma_c(\ker d\bt)+\Gamma_c(\ker d\bs),$$
and the analogue equation holds for $\bt$. As a consequence, $(M,\cF_{\Pi_M})\simeq_{ME} (N,\cF_{\Pi_N})$.
\end{ex}

\begin{cor}\label{cor:PoisME}
If two Poisson manifolds are Morita equivalent
\cite{xuME} then their singular foliations   are Hausdorff Morita equivalent.
\end{cor}
\begin{proof}
Two Poisson manifolds are Morita equivalent if they are related by a complete full dual pair with simply connected fibers. Hence the statement follows from Ex. 
\ref{ex:fulldualpair}.
\end{proof}

\subsubsection{Morita equivalence for Lie groupoids and Lie algebroids} 
We   review briefly Morita equivalence for Lie groupoids and Lie algebroids, recalling a few notions given in \cite{MdMkGrd} and \cite{GinzburgGrot} (see also \cite{LauMatXu}):

\begin{defi}\label{def:MEgroid}
Two Lie groupoids $G\rightrightarrows M$ and $H\rightrightarrows N$ are \textbf{Morita equivalent} if  there exists a (Hausdorff) manifold $P$, and two surjective submersions $\pi_M:P\fto M$ and $\pi_N:P\fto N$ such that $\pi^{-1}_M G\cong \pi_N^{-1}H$.

In this case, we call $(P,\pi_M,\pi_N)$ a Morita equivalence between $G$ and $H$.
\end{defi}
{Alternative characterisations of
 Morita equivalence for Lie groupoids are recalled in  Appendix \ref{sec:MEGrpd}.}

The definition of (weak) Morita equivalence for Lie algebroids is due to Viktor Ginzburg  \cite{GinzburgGrot}:

\begin{defi}\label{def:MEalgoid}
Consider Lie algebroids $A_M$ and $A_N$ over the manifolds $M$ and $N$ respectively. We say they are \textbf{Morita equivalent} if there exists a manifold $P$ and two surjective submersions $\pi_M\colon P\fto M$ and $\pi_N:P\fto N$ with  simply connected fibres such that $\pi_M^{-1}( A_M)\cong \pi_N^{-1}(A_N)$ as Lie algebroids over $P$.
\end{defi}
This definition can be motivated by the following proposition:
 
\begin{prop}\label{prop:MEalggroids}
(i) If $G_M$ and $G_N$ are Morita equivalent {Hausdorff} Lie groupoids with source simply connected fibres, then their Lie algebroids are Morita equivalent.

(ii) If $A_M$ and $A_N$ are Morita equivalent integrable Lie algebroids, then the  source simply connected Lie groupoids integrating them are Morita equivalent.
\end{prop}
\begin{proof}
(i): By corollary \ref{cor:SConnMEGrpd2} we {obtain the existence of} a Morita equivalence with source simply connected fibres, then using part (1) of lemma \ref{GtoA} we get the desired result.

(ii): is clear by part (2) of lemma \ref{GtoA}.
\end{proof}

\begin{rem}
The essential difference between  Morita equivalence for Lie algebroids and {Hausdorff Morita equivalence} for singular foliations (Def. \ref{def:MEalgoid} and Def. \ref{def:defMEfol}) is that the former requires \emph{simply} connected fibres whereas the latter only connected fibres. This difference is reflected at the groupoid level too, as we now explain.

On the one hand, given two Morita equivalent integrable Lie algebroids, their source simply connected Lie groupoids are Morita equivalent, see proposition \ref{prop:MEalggroids}(ii). On the other hand, singular foliations  also have an associated groupoid, namely the holonomy groupoid defined by Androulidakis and Skandalis  \cite{AndrSk}. In theorem \ref{thm:MEfolgroids} we will show that if two singular foliations are {Hausdorff} Morita equivalent, then their holonomy groupoids are also Morita equivalent. But the holonomy groupoid of a foliation does not have  simply connected fibres in general: on the contrary, it is an adjoint groupoid.
\end{rem}

\subsubsection{{Examples  from Morita equivalence of Lie groupoids and Lie algebroids}}\label{subsubsec:related}

Recall that Lie groupoids give rise to Lie algebroids, which in turn give rise to singular foliations (see example \ref{ex:bisub.grpd}).

\begin{prop}\label{prop: implications}
  Let $G\soutar M$ and $H\soutar N$ source connected {Hausdorff} Lie groupoids, denote their Lie algebroids by $A_M$ and $A_N$, and denote by    $\CF_M=\#(\CG_c(A_M))$ and $\CF_N=\#(\CG(A_N))$ the corresponding singular foliations. Each of the following statements implies the following one (i.e.  $(i)\Rightarrow (ii)   \Rightarrow (iii)$):
\begin{enumerate}
\item[(i)]  $G$ and $H$ are Morita equivalent,
\item[(ii)]  there exists a  manifold $P$ and surjective submersions with connected fibres $\pi_M\colon P\fto M$ and $\pi_N:P\fto N$ satisfying $\pi_M^{-1}A_M\cong \pi_N^{-1} A_N$,
\item[(iii)] the foliated manifolds $(M,\CF_M)$ and $(N,\CF_N)$ are Hausdorff Morita equivalent.
\end{enumerate}    
\end{prop}
\begin{proof}
For the first implication, use corollary \ref{cor:SConnMEGrpd2} to get $P$, then we use   twice   part (i) of lemma \ref{GtoA}.
 The second implication   follows using twice  lemma \ref{lem:pullbackalgfol}. 
\end{proof}

\begin{rem}
If $A_M$ and $A_N$ are Morita equivalent Lie algebroids then their singular foliations are Morita equivalent. Indeed the fibres of the maps appearing in Def. \ref {def:MEalgoid} are in particular connected, so the above proposition applies.\end{rem}

\begin{ex} { 
Consider two {connected} Lie groups $G_1,G_2$ acting freely and properly on a manifold $P$ with commuting actions. The transformation groupoids $G_2\ltimes (P/G_1)$ and $G_1\ltimes (P/G_2)$ are Morita equivalent, since under the quotient maps {$P\to P/G_i$} they pull back
to the transformation groupoid  $(G_1\times G_2)\ltimes P$. Hence applying proposition \ref{prop: implications} we obtain an alternative proof of the statement of 
  corollary \ref{cor:action}.}
\end{ex}
 
\section{Morita equivalent holonomy groupoids}\label{sec:hol}  

{We recall how, canonically associated to a singular foliation, there is a open topological groupoid, called \emph{holonomy groupoid}. In  \S \ref{subsec:hol3}, building on   \S \ref{sec:hol2}, we can prove the main statement of the paper (theorem. \ref{thm:MEfolgroids}): Hausdorff Morita equivalence of singular foliations implies the Morita equivalence of their holonomy groupoids. This will allow us to obtain further invariants in \S \ref{subsec:holtrafo} and \S
\ref{subsec:furtherinv}.}  

\subsection{Holonomy groupoids}\label{sec:hol1}

{We review the   construction of the holonomy groupoid of a foliated manifold due to Androulidakis-Skandalis, following \cite[\S 2, \S 3.1]{AndrSk}.  In this whole subsection we fix a foliated manifold $(M,\CF)$.}

\subsubsection{{Bisubmersions and bisections}}

{The holonomy groupoid is constructed using bisubmersions, which we recalled in Def. \ref{def:bisub}.}

 \begin{defi}
Let $(U,\bt_U,\bs_U)$ and $(V,\bt_V,\bs_V)$ be bisubmersions.
\begin{enumerate}
	\item[(i)] The {\bf inverse} bisubmersion of $U$ is $U^{-1}:=(U,\bs_U,\bt_U)$, the bisubmersion obtained interchanging source and target.
	\item[(ii)] Let $W:=U{}_{\bs_U} \!\times_{\bt_V} V$, then $U\circ V:=(W,\bt_U,\bs_V)$ is called the {\bf composition} bisubmersion of $U$ with $V$.
\end{enumerate}
\end{defi} 

\begin{defi}
Given a foliated manifold $(M,\CF)$ and two bisubmersions {$U$ and $V$,} a smooth map $f\colon U\fto V$ is called a {\bf morphism} of bisubmersions if it commutes with the source and the target maps of $U$ and $V$.
\end{defi}

\begin{defi} Consider a foliated manifold $(M,\CF)$, a bisubmersion  $(V,\bt,\bs)$  and $x\in \bs(V)$. 
\begin{enumerate}
\item[(i)] A {\bf bisection} at $x$ consists of a local $\bs$-section $\sigma\colon M'\fto V$, where $M'$ is a neighbourhood of $x$ in $\bs(V)$, such that the image of $\sigma$ is transverse to the fibres of $\bt$.
\item[(ii)] {Given a  diffeomorphism $\phi$ between open subsets of $M$, a bisubmersion $(V,\bt,\bs)$ is said to {\bf carry} $\phi$ at $v\in V$ if  there exists a bisection $\sigma$ through $v$ such that $\phi=\bt\circ\sigma$.}
\end{enumerate}
\end{defi}

The following propositions are to be found in \cite[\S 2.3]{AndrSk}.

\begin{cor}\label{carry}
Let $(U,\bt_U,\bs_U)$ and $(V,\bt_V,\bs_V)$ be bisubmersions and $u\in U$, $v\in V$ be such that $\bs_U(u)=\bs_V(v)=:x$. Then:
\begin{enumerate}
\item[(i)] If there is a local diffeomorphism carried both by $U$ at $u$ and by $V$ at $v$, there exists an open neighbourhood $U'$ of $u$ in $U$ and a morphism $f:U'\fto V$ such that $f(u)=v$.
\item[(ii)] If there is a morphism $g:V\fto U$ such that $g(v)=u$ then there exists an open neighbourhood $U'$ of $u$ in $U$ and a morphism $f:U'\fto V$ such that $f(u)=v$.
\end{enumerate}
\end{cor}

\begin{prop}\label{path} Given $x_0\in M$, let $X_1,\dots ,X_n\in \CF$ be vector fields whose classes in the fibre $\CF_{x_0}$ form a basis. For $v=(v_1,\dots, v_n)\in \RR^n$, put $\varphi_v= exp(\Sigma_i v_i X_i)$, {where $exp$ denotes the time one flow.}
	
Put $W=\RR^n\times M$, $\bs(v,x)=x$ and $\bt(v,x)=\varphi_v(x)$.
\begin{enumerate}
\item[(i)] There is a neighbourhood $U\subset W$ of ${(0,x_0)}$ such that $(U,\bt,\bs)$ is a bisubmersion.
\item[(ii)] Let $(V,\bt_V,\bs_V)$ a bisubmersion and $v_0\in V$. Assume that $s(v_0)=x_0$ and that $V$ carries the identity diffeomorphism at $v_0$. There exists an open neighbourhood $V'$ of $v_0$ on $V$ and a submersion $g:V'\fto U$ which is a morphism of bisubmersions and $g(v_0)=(0,x_0)$. 
\end{enumerate}
\end{prop}

\begin{defi}
A bisubmersion as in proposition \ref{path}  (i), {when it has $\bs$-connected fibres,} is called {\bf path holonomy bisubmersion}.
\end{defi}

\begin{rem}\label{rem:invertmorph}
{Neighbourhoods of points of the form $(0,x_0)$ in  path holonomy bisubmersions can be embedded in any bisubmersion that carries the identity diffeomorphism. This follows from  {proposition} \ref{path} (ii) and corollary \ref{carry} (ii).}
\end{rem}

\subsubsection{{The holonomy groupoid: construction and properties}}
\label{sec:holconstr}
 
\begin{defi}\label{def:adapted}
Let $\CU=(U_i,\bt_i,\bs_i)_{i\in I}$ be a family of bisubmersions.
\begin{enumerate}
\item[(i)] A bisubmersion $(V,\bt,\bs)$ is said to be {\bf adapted} to $\CU$ at $v_0\in V$ if there exists an open neighbourhood $V'\subset V$ of $v_0$, an element $i\in I$ and a morphism of bisubmersions $V'\fto U_i$. We say that $V$ is adapted to $\CU$ if it is adapted for all point.
\item[(ii)] We say that $\CU$ is an {\bf atlas} if
\begin{itemize}
\item {$\bigcup_{i\in I} {\bs_i}(U_i)= M$},
\item the inverse and the composition of every element in $\CU$ is adapted to $\CU$.
\end{itemize}
\item[(iii)] Let $\CU$ and $\mathcal{V}$ two atlases. We say that $\CU$ is {\bf adapted} to $\mathcal{V}$ if every element in $\CU$ is adapted to $\mathcal{V}$. We say that they are equivalent if they are adapted to each other.
\end{enumerate}
\end{defi}

\begin{prop}
[{\bf Groupoid of an atlas}]
\label{prop:groidatlas} Let $(M,\CF)$ be a foliated manifold and $\CU=(U_i,\bt_i,\bs_i)$ an atlas of bisubmersions for $\CF$.
\begin{enumerate}
\item[(i)] On $\coprod_{i\in I} U_i$ there is an equivalence relation $\sim$ given by: $U_i\ni u\sim v\in U_j$ if there exists a local morphism of bisubmersions from $U_i$ to $U_j$ mapping $u$ to $v$.

Denote $G=G(\CU)$ the quotient by this equivalence relation, and $Q=(q_i)_{i\in I}\colon \sqcup_i U_i \fto G$ the quotient map.
\item[(ii)] There are maps $\bt,\bs\colon G\fto M$ such that $\bs\circ q_i=\bs_i$ and $\bt\circ q_i=\bt_i$.
\item[(iii)] There is a groupoid structure on $G$ with set of objects $M$, source and target maps $\bs$ and $\bt$ defined above and such that $q_i(u)q_j(v)=q_{U_i\circ U_j}(u,v)$.
\end{enumerate}
\end{prop}
\begin{rem}
{By corollary \ref{carry} (i),} the equivalence relation can be also stated as: $u\sim v$ if and only if there are local bisections through $u$ and $v$ carrying the same diffeomorphism.
\end{rem}

Given a foliated manifold $(M,\CF)$ and an atlas of bisubmersions $\CU=\{U_i : i\in I\}$  for $\CF$, we endow $G(\CU)$ with the quotient topology, i.e.  the smallest topology that makes the quotient map: $Q\colon\sqcup_i U_i \fto G(\CU)$ continuous.

\begin{lem}\label{lem:openmap}
Given a foliated manifold $(M,\CF)$ and  an atlas of bisubmersions $\CU=\{U_i \st i\in I\}$ for $\CF$, the quotient map $Q\colon\sqcup_i U_i \fto G(\CU)$ is open.
\end{lem}
\begin{proof}
Let $\CU$ be an atlas of bisubmersions, we will prove that given an open subset $A$   of $\sqcup_i U_i$, the preimage $Q^{-1}(Q(A))$ is open.

Take $x\in Q^{-1}(Q(A))$. {Denote by $U\in\CU$ the bisubmersion} such that $x\in U$. There exists $y\in A$ such that $Q(x)=Q(y)$. {Notice that $A$ itself is a bisubmersion. By the definition of the equivalence relation in proposition \ref{prop:groidatlas},} 
there is a neighbourhood $U'\subset U$ of $x$ and a  morphism of bisubmersions $f\colon U'\fto {A}$ sending $x$ to $y$. {Hence $Q(U')\subset Q(A)$, or in other words }$U'\subset Q^{-1}(Q(A))$, therefore $x$ is an interior point of $Q^{-1}(Q(A))$.
\end{proof}

\begin{lem}\label{lem:adapt}
If the atlas $\CU_1$ is adapted to $\CU_2$:
\begin{enumerate}
\item[(i)] there is a {canonical injective morphism of topological} groupoids $\varphi\colon G(\CU_1)\fto G(\CU_2)$,
\item[(ii)] $\varphi$ is surjective if and only if $\CU_2$ is adapted to $\CU_1$. {In that case} $\varphi$ is an isomorphism of topological groupoids.
\end{enumerate}
\end{lem}
{The map $\varphi$ is induced by morphisms from bisubmersions.
More precisely:} 
 there is a well defined map $\widehat{\varphi}\colon \sqcup_{U\in \CU_1}U \fto G(\CU_2)$, given by   {$u\mapsto [f(u)]$ where $f$ is any  morphism of bisubmersions from a neighbourhood of $u$ to a bisubmersion in $\CU_2$}. This maps factor through the quotient map $Q_1$, {yielding $\varphi$.}
 \begin{equation} \label{diag:cont1}
\begin{tikzcd}\sqcup_{U\in \CU_1} U \arrow[d,"Q_1"] \arrow[dr,"\widehat{\varphi}"] & \\
G(\CU_1) \arrow[r,"\varphi"]   & G(\CU_2).
\end{tikzcd}
\end{equation}
 
\begin{defi}\label{def:pathholatlas}
Let $(M,\CF)$ be a foliated manifold. {A {\bf path holonomy atlas} is an atlas generated by   a family of path holonomy bisubmersions ${(U_i,\bt_i,\bs_i)}_{i\in I}$ such that $\cup\bs_i(U_i)=M$}. 
\end{defi}

{The following lemma implies easily corollary \ref{cor:phatlas} and corollary \ref{cor:adapted}, which together are the content of  \cite[Examples 3.4(3)]{AndrSk}.} 

\begin{lem}\label{lem:smallelements} {Let $(M,\cF)$ be a foliated manifold.
Let $U\subset \RR^n\times M$ be a path holonomy bisubmersion, and $\CV$ an atlas of bisubmersions for $\cF$. Then $U$ is adapted to $\CV$.}
\end{lem}
\begin{proof} We have to show that around any point of $U$ there is a locally defined morphism of bisubmersions to an element of $\CV$, see definition \ref{def:adapted}.

Let $u=(v,x)\in U$. {If $v=0$ we can simply apply Remark \ref{rem:invertmorph}, so in the following we assume $v\neq 0$.}
Denote by $X_1,\dots,X_n$ the vector fields in $\cF$ used to construct the path holonomy bisubmersion $U$.  Extend $v$ to a basis $v^1:=v, v^2,\dots,v^n$ of $\RR^n$, and consider the path holonomy bisubmersion $\tilde{U}$ given by the local generators $\sum_i v_i^1X_i,\dots, \sum_i v_i^nX_i$ of $\cF$. The points $u\in U$ and $((1,0,\dots,0),x)\in \tilde{U}$ are equivalent by {corollary \ref{carry}(i)}, since the constant bisections through them carry the same diffeomorphism. Hence in the rest of the proof we can assume that $v=(1,0,\dots,0)$.

Since $X_1$ is compactly supported and hence complete, we can consider the path $\gamma\colon [0,1]\to M, \gamma(h)=exp_x(hX_1)$. For every $h\in [0,1]$, apply the diffeomorphism $exp(hX_1)$ to $X_1,\dots,X_n$. This yields    elements of $\cF$, the first one being $X_1$, which form a generating set for $\cF$ near $\gamma(h)$. Denote by $U_h$  the path holonomy bisubmersion  they give rise to. 

By remark \ref{rem:invertmorph} there exists an open neighborhood $U_h'$ of $(0,\gamma(h))$ and a morphism of bisubmersions from $U_h'$ to a bisubmersion in $\CV$.
Shrinking $U_h'$ if necessary, we can assume that it is of the form ${B_{{r_h}}}\times M'_h$ where  $B_{{r_h}}\subset \RR^n$ is the open ball with radius $r_h$ and $M'_h\subset M$.

By the compactness of $[0,1]$, there are finitely many $h_1,\dots,h_k\in [0,1]$ such that $M'_{h_1},\dots,M'_{h_k}$ cover the image of $\gamma$. Hence
there is a positive integer $N$ such that, for all $h\in [0,1]$, the point $(\frac{1}{N}v,\gamma(h))$ is contained in one of the $U_{h_i}'$.
The composition 
\begin{equation}\label{eq:comp}
  \left(\frac{v}{N}, \gamma\Big(\frac{N-1}{N}\Big)\right)\circ\dots\circ \left(\frac{v}{N}, \gamma\Big(\frac{1}{N}\Big)\right)\circ\left(\frac{v}{N},x\right)
\end{equation} is well-defined\footnote{For instance, $\bt((\frac{v}{N},x))=
exp_x(\frac{1}{N}X_1)=\gamma(\frac{1}{N})$.}.
 Further, it is equivalent to $u=(v,x)\in U$ since the constant bisections through {
$u$ and through the composition \eqref{eq:comp}} 
 both carry the  diffeomorphism  $exp(X_1)$. Since each of the elements we are composing in \eqref{eq:comp} lies in the domain of a morphism of bisubmersions  to a bisubmersion in $\CV$, the composition also does.
\end{proof}

\begin{cor}\label{cor:phatlas}
\begin{enumerate}
\item [(i)] A {path holonomy  atlas} is adapted to any  atlas. 
\item [(ii)] {Any path holonomy atlas defines the same topological groupoid.} 
\end{enumerate}
\end{cor}
\begin{proof}
{(i) 
A path holonomy atlas consists of  finite compositions of path holonomy bisubmersions. Hence the statement follows from lemma \ref{lem:smallelements}.}

(ii)  {An immediate consequence of part (i) is that any pair of path holonomy atlases are adapted to each other, therefore they define the same  topological groupoid.}
\end{proof}

\begin{defi}\label{def:holgroidph}
Let $(M,\CF)$ be a foliated manifold. The groupoid over $M$ associated to a path holonomy atlas is called {\bf holonomy groupoid} and  is denoted $H(\CF)$.
\end{defi}

\begin{cor}\label{cor:adapted}
There exists a {canonical} injective {morphism of topological groupoids}
\[\varphi\colon H(\CF)\fto G\]
where $G$ is any groupoid given by an atlas of bisubmersions for $\CF$.
\end{cor}
\begin{proof}
This follows from   corollary \ref{cor:phatlas} (i)  {applied to the atlas defining $G$, and from   {lemma \ref{lem:adapt} (i)}.}
\end{proof}

\begin{lem}\label{lem:sconnholgroupid}
The holonomy groupoid is source connected.
\end{lem}
\begin{proof}

{Take a family of path holonomy bisubmersions $\CU:=\{U_i\}_{i\in I}$ such that $\{{\bs_i}(U_i)\}_{i\in I}$ covers $M$ and $\{0\}\times {\bs_i}(U_i)\subset U_i$. ({It exists by   lemma \ref{path} (i)}.) Taking finite compositions and inverses of these elements we {obtain} a path holonomy atlas $\widehat{\CU}$. Denote by $Q\colon \sqcup_{U\in \widehat{\CU}} U\fto H(\CF)$ the {(surjective)} quotient map.}

{ Note that for any $U\in\CU$, any point of $Q(U)$ \emph{can be connected to the identity through {a continuous path in} an $\bs$-fiber of $H(\cF)$}. Now we  prove {the same statement for any point of} $Q(U_k\circ\dots\circ U_1)$, where {$k\ge 2$} and $U_k,\dots ,U_1\in \CU$.}

By induction, suppose {that the statement holds for all points of $Q(U_{k-1}\circ\dots\circ U_1)$}.
Take  ${u:=}u_k\times\dots \times u_1 \in U_k\circ \dots \circ U_{1}$, {and denote $p:=\bs(u_k)\in M$.}
Because $U_k$ is source connected   there exists a curve $\gamma(t)$ in {a source fibre of} $U_k$ 
{joining $u_k$ with $(0,p)$.} 
Then $Q(\gamma(t)\times\dots \times u_1)$
{is a curve in an $\bs$-fibre of $H(\cF)$
that connects $Q(u)$ with an element of $Q(U_{k-1}\circ\dots\circ U_1)$. Hence the statement holds of $Q(u)$.}
\end{proof}

\subsection{Morita equivalent holonomy groupoids: the  case of projective foliations}

In this short subsection we show that, under certain assumptions, Morita equivalence of holonomy groupoids implies {Hausdorff} Morita equivalence of singular foliations. The converse implication will be addressed in  {\S \ref{subsec:hol3}}. {In the case of regular foliations, these results are in  agreement with   Haefliger's approach to Morita equivalence \cite[\S 1.5]{HaefligerHolonomieClassifiants}.}

We saw in \S \ref{subsec:firstinv} that projective foliations have an associated almost injective Lie algebroid. The latter is always integrable {\cite{DebordJDG} to a Lie groupoid. Indeed},   a singular foliation is projective if{f} its holonomy groupoid is a Lie groupoid \cite{AZ5}.  

\begin{prop}\label{prop:MEholtoMEfol}
{Let $(M,\cF_M)$ and $(N,\cF_N)$ be projective singular foliations. If the holonomy groupoids $H(\cF_M)$ and $H(\cF_N)$ are} {Hausdorff}
%\footnote{See \Haus}
 and are Morita equivalent, then the two singular foliations are Hausdorff Morita equivalent.
\end{prop}

\begin{proof}
The hypothesis on $\cF_M$ assures that $H(\cF_M)$ is a source connected {Hausdorff} Lie groupoid, whose underlying foliation is $\cF_M$, and similarly for $\cF_N$. Hence one can apply {proposition \ref{prop: implications}}. 
\end{proof}

\begin{rem}
We do not know if this statement holds for arbitrary singular foliations. To establish such a result, one first needs to describe precisely how a singular foliation can be recovered from its holonomy groupoid. To do this, the holonomy groupoid has to be viewed as a diffeological groupoid rather than just as a topological one.
This  issue   is being addressed in \cite{AZ4}.
\end{rem}

\subsection{Pullbacks of foliations and their holonomy groupoids}\label{sec:hol2}

In this subsection we show that, for suitable maps, 
the holonomy groupoid of the pullback foliation is the pullback of the 
holonomy groupoid of the foliation. {From this in \S \ref{subsec:hol3} we will easily derive the main result of the paper, namely theorem. \ref{thm:MEfolgroids}.}

\subsubsection{An isomorphism of {topological} groupoids}

We prove the following isomorphism of topological groupoids:
\begin{thm}\label{thm:pullbackgroid}
Let $(M,\CF)$ be a foliated manifold and $\pi\colon P\fto M$ a surjective submersion {with connected fibres}.
Then there is a canonical isomorphism of topological groupoids $$H(\pi^{-1}(\CF))\cong \pi^{-1}(H(\CF)).$$
\end{thm} 
{The pullback of a topological groupoid  is defined as in Def. \ref{def:pullbackgroid}}. For the sake of exposition, we first sketch a proof of theorem. \ref{thm:pullbackgroid} in the case of regular foliations.

\begin{proof}[Proof of theorem. \ref{thm:pullbackgroid} for regular foliations]
Assume that $\CF$ is a regular foliation, then $\pi^{-1}\CF$ is also a regular foliation. In this case the holonomy groupoid is given by holonomy classes of paths. Define the {groupoid morphism}

\begin{align*}
 \varphi\colon H(\pi^{-1}(\CF))&\fto \pi^{-1}(H(\CF))\\
  [\tau]\;\;&\mapsto(\tau(1),[\pi(\tau)],\tau(0)).
\end{align*}

We show that this map is  injective. Take two paths $\tau,\tilde{\tau}$ in leaves of $P$ with the same initial point $p_0$ and  final point $p_1$.
For $i=0,1$, if $\Sigma_i$ is a transversal to $\pi^{-1}(\CF)$ at $p_i$ then $\pi(\Sigma_i)$ is a transversal of $\CF$ at $\pi(p_i)$. Let $\Phi, \tilde{\Phi}\colon\Sigma_0\fto\Sigma_1$ be the holonomy maps given by $\tau$ and $\tilde{\tau}$ respectively, and $\phi,\tilde{\phi}\colon \pi(\Sigma_0)\fto \pi(\Sigma_1)$ the holonomy maps given by $\pi(\tau)$ and $\pi(\tilde{\tau})$. The following diagram commutes:
 \[\begin{tikzcd}
\Sigma_0 \arrow[d,"\pi"] \arrow[r,"\Phi"] & \Sigma_1 \arrow[d,"\pi"] \\
\pi(\Sigma_0) \arrow[r,"\phi"] & \pi(\Sigma_1), \\ 
\end{tikzcd}\]
and the analog diagram for $\tilde{\Phi}, \tilde{\phi}$ too.
The vertical maps $\pi\colon \Sigma_i \fto \pi(\Sigma_i)$ are diffeomorphisms (notice that the codimensions of $\CF$ and $\pi^{-1}\CF$ are equal).
{Hence if $\pi(\tau)$ and $\pi(\tilde{\tau})$ have the same holonomy, i.e.  $\phi=\tilde{\phi}$, then $\Phi=\tilde{\Phi}$}.

To prove the surjectivity of $\varphi$, take $(p,[\gamma],q)\in \pi^{-1}(H(\CF))$ where $p,q\in P$ and $\gamma$ is a curve   {in a leaf} of $\CF$ that connects $\pi(q)$ with $\pi(p)$. 
The hypotheses on $\pi$ imply that $\text{Pr}_1\colon \gamma^{*}P:=[0,1] {}_\gamma \!\times_\pi P \fto [0,1]$ is a surjective submersion with connected fibres, hence  $\gamma^{*}P$ is a connected manifold and  therefore a path connected space. Take a curve  $\sigma\colon [0,1]\fto \gamma^{*}P$ that connects $(0,q)$ with $(1,p)$.
We have the following commutative diagram:
\[\begin{tikzcd}
\left[0,1 \right] \arrow[r,"\sigma"] &\gamma^{*}P=[0,1] {}_\gamma \!\times_\pi P \arrow[d,"\text{Pr}_1"] \arrow[r,"\text{Pr}_2"] & P \arrow[d,"\pi"] \\
&\left[0,1 \right] \arrow[r,"\gamma"] & M.
\end{tikzcd}\]
%where $\text{Pr}_2$ is the second projection. 
The curve
$\widehat{\gamma}:= \text{Pr}_2\circ \sigma$ lies in a leaf of $\pi^{-1}(\cF)$ and joins $q$ with $p$. Since $\text{Pr}_1 \circ\sigma\colon [0,1]\fto[0,1]$ is a continuous and surjective function homotopic to the identity, from the commutativity of the  diagram it follows that $\pi\circ \widehat{\gamma}$ and $\gamma$ are homotopy equivalent and so holonomy equivalent. {Hence} $\varphi([\widehat{\gamma}])=(p,[\gamma],q)$,   proving that $\varphi$ is surjective and therefore bijective. 
\end{proof}

We now turn to the proof of theorem. \ref{thm:pullbackgroid}. {The first step is to state and prove proposition \ref{super},} which requires some preparation. We first focus on
pullbacks of atlases {of bisubmersions}, which are relevant for
the l.h.s. of the isomorphism claimed there.
We state first   \cite[lemma 2.3]{AndrSk}:

\begin{lem}\label{lem:compbisub} Let $(U, \bt, \bs)$ be a bisubmersion of $(M,\CF)$ and $\pi \colon W\fto U$ be a submersion. Then $(W, \bt\circ \pi, \bs\circ \pi)$ is a bisubmersion  for $\cF$.
\end{lem}
 The following lemma allows us to pull back bisubmersions.

\begin{lem}\label{pullbi}
Let $(M, \CF)$ be a foliated manifold, $(U, \bt, \bs)$ a bisubmersion for $\CF$ and $\pi\colon P\fto M$ a surjective submersion. Consider the preimages $P^\bs:=\pi^{-1}(\bs(U))$ and  $P^\bt:=\pi^{-1}(\bt(U))$. Define $$\pi^{-1}(U):=P^\bt {}_{\pi} \!\times_{\bt} U {}_\bs \!\times_\pi P^\bs.$$ Let $\tau,\sigma\colon \pi^{-1}(U)\fto P$ be the  projections {onto the first and third component. Then $(\pi^{-1}(U),\tau,\sigma)$} is bisubmersion for $\pi^{-1}(\CF)$.
\end{lem}
\begin{proof}
The following 
diagram commutes:
\[
\begin{tikzcd}
\pi^{-1}(U) \arrow[d,"\text{Pr}_U"] \arrow[r,shift right=.3em,swap,"{\sigma}"] 
\arrow[r,shift left=.3em,"\tau"] & P \arrow[d,"\pi"] \\
U \arrow[r,shift right=.3em,swap,"\bs"] \arrow[r,shift left=.3em,"\bt"] & M
\end{tikzcd}
\]
Moreover, since $\pi$ is a  submersion one can prove that  $\tau$ and $\sigma$ are submersions. {By the same reason $\text{Pr}_U$ is a submersion, and applying lemma \ref{lem:compbisub} {to it} we obtain $(\bt\circ \text{Pr}_U)^{-1}(\cF)=
(\bs\circ \text{Pr}_U)^{-1}(\cF)$.}
 Using the commutativity of the diagram  we are done.
\end{proof}

\begin{defi}
We call the bisubmersion $\pi^{-1}(U)$ given in lemma \ref{pullbi} the \textbf{pullback bisubmersion} of $U$.
\end{defi}
\begin{lem}\label{lem:pullbackatlas}
Let $\CU$ be an atlas of bisubmersions for $\CF$. Then $\pi^{-1} \CU:=\{\pi^{-1}(U) \st U\in\CU\}$ is an atlas of bisubmersion for $\pi^{-1}(\CF)$.
\end{lem}
\begin{proof}
It is clear that the union of the elements of $\pi^{-1} \CU$ covers $P$. We now check that the compositions of elements in $\pi^{-1} \CU$ are adapted to $\pi^{-1} \CU$. 
To do so, take $U_2,U_1\in \CU$. Note that we have a canonical morphism of {bisubmersions} 
\begin{equation}\label{eq:morphcomp}
\pi^{-1}(U_2)\circ \pi^{-1}(U_1)  \fto \pi^{-1}(U_2\circ U_1);\;\;(p,u_2,a)\times(a,u_1,q)\mapsto(p,u_2\times u_1,q).  
\end{equation}
Moreover, since $\CU$ is an atlas, at   each point $u_2\times u_1$ the bisubmersion $U_2\circ U_1$ is adapted to some $U_{u_1\times u_2}\in \CU$. This means that there is a small neighbourhood $V'\subset U_1\circ U_2$ containing $u_2\times u_1$, and a morphism of bisubmersions $f\colon V'\fto U_{u_1\times u_2}$. {Composing a suitable restriction of the morphism \eqref{eq:morphcomp} with the natural ``lift'' of $f$ we obtain
 a morphism of bisubmersions}  $$(p,u_2,a)\times(a,u_1,q)\mapsto(p,f(u_2\times u_1),q)\in \pi^{-1}(U_{u_1\times u_2})$$
{into an element of $\pi^{-1}\CU$.}
{For inverses of elements in $\pi^{-1} \CU$ one  proceeds  similarly}.
\end{proof} 

\begin{prop}\label{super}
Let $\CU$ be an atlas of bisubmersions on a foliated manifold $(M,\CF)$, and denote by $G(\CU)$ the groupoid given by $\CU$. Let $\pi\colon P\fto M$ be a surjective submersion,  $\pi^{-1} \CU$ the pullback atlas, and denote by $G(\pi^{-1} \CU)$ the groupoid of this atlas. Then there is a canonical {isomorphism of topological groupoids} 
$$G(\pi^{-1} \CU)\cong\pi^{-1}(G(\CU)).$$
\end{prop}

\begin{proof}
 {The quotient map $Q\colon\sqcup_{U\in\CU} U \fto G(\CU)$ lifts to
a canonical map 
\begin{equation}\label{map:pullbackatlas}
Id\times Q\times Id\colon \left(\sqcup_{U\in\CU} \pi^{-1}U\right) \fto \pi^{-1} (G(\CU)),\hspace{0.2in} (p,u,q)\mapsto (p,[u],q),
\end{equation}
where we denote  $[u]:=Q(u)$.}
We will show that this map factor through the projection 
map associated to the atlas $\pi^{-1}\CU$, determining 
a map  $\Phi\colon G(\pi^{-1}\CU)\fto \pi^{-1}(G(\CU))$,
which moreover   is an isomorphism of topological groupoids. 
\begin{equation}
\begin{tikzcd}\label{diag:iQi}
\left(\sqcup_{U\in\CU} \pi^{-1}U\right) \arrow[dr,"Id\times Q\times Id"] \arrow[d,"Q_\pi"]&  \\
 G(\pi^{-1}\CU) \arrow[r,dashed,"\Phi"] & \pi^{-1} (G(\CU))
\end{tikzcd}  
\end{equation}

{The fact that $\Phi$ is well-defined and injective follow from the claim below (respectively, from the implications ``$\Rightarrow$'' and ``$\Leftarrow$'').} The surjectivity of $\Phi$ is clear because the map $Id\times Q\times Id$ given in \eqref{map:pullbackatlas} is surjective. The fact that $\Phi$ is a homeomorphism holds because both $Q_\pi$   and $Id\times Q\times Id$ are open maps, {by lemma \ref{lem:openmap}. The map $\Phi$ is a groupoid morphism as a consequence of proposition \ref{prop:groidatlas} {(iii)} and of the morphism of bisubmersions \eqref{eq:morphcomp}. Hence we are left with proving the following claim for all $ (p_0,u_0,q_0), (p_0,v_0,q_0)\in\sqcup_{U\in\CU} \pi^{-1}U$.}\\

\noindent\underline{{ Claim:}}   {\it $Q_\pi(p_0,u_0,q_0)=Q_\pi(p_0,v_0,q_0)$ in $G(\pi^{-1}\CU)$ if and only if $[u_0]=[v_0]$ in $G(\CU)$}.

``$\Rightarrow$'': 
{By assumption} 
there exist bisections $\sigma_u$ and $\sigma_v$ {through}  $(p_0,u_0,q_0)$ and $(p_0,v_0,q_0)$ respectively, and carrying the same diffeomorphism {of $P$}. Since $\pi$ is a submersion, there exists a neighbourhood $W$ of $\pi(q_0)\in M$ and a $\pi$-section $q\colon W\fto P$, such that $q(\pi(q_0))=q_0$. Finally $\text{Pr}_2 \circ\sigma_u\circ q $ is a bisection {through} $u_0$ carrying the same diffeomorphism {of $M$} as the bisection $\text{Pr}_2\circ\sigma_v\circ q $ {through} $v_0$. Therefore $[u_0]=[v_0]$ in $G(\CU)$.

``$\Leftarrow$'': Let {$ u_0\in U$ and $v_0\in V$ be equivalent points of} $\sqcup_{U\in \CU} U$, and let $p_0,q_0\in P$  lie in the fibre of $\bt(u_0)=\bt(v_0)$ and $\bs(u_0)=\bs(v_0)$ respectively. Then there exists a neighbourhood $U'$ of $u_0$ inside $U$ and a morphism {of bisubmersions} $f\colon U'\fto V$ {such that $f(u_0)=v_0$}.  
{Lifting it}  we get a morphism of bisubmersions $$\widehat{f}\colon \pi^{-1}U'\fto \pi^{-1}V\hspace{.1in} ;\hspace{.1in} (p,u,q)\mapsto (p,f(u),q)$$ such that $\widehat{f}(p_0,u_0,q_0)=(p_0,v_0,q_0)$. This shows that $Q_\pi(p_0,u_0,q_0)=Q_\pi(p_0,v_0,q_0)$ in $G(\pi^{-1}\CU)$.
\end{proof}

{We now take the second step for the proof of theorem. \ref{thm:pullbackgroid}.}
{
 For all path holonomy bisubmersions $U$ of $\cF$, as seen in lemma \ref{pullbi}, the pullback $\pi^{-1}(U)$ is a bisubmersion for $\pi^{-1}(\cF)$, but it is not a path holonomy bisubmersion in general (its image under the source map is a saturated open set of $P$). However its restriction to small neighbourhoods in $P$ is isomorphic to a path holonomy bisubmersion, as a consequence of \cite[proposition 2.10 b]{AndrSk}. This fact underlies the following proposition.}

\begin{prop}\label{prop:injmap}
{Let $\pi\colon P\fto M$ be a surjective submersion { with connected fibres} and $\cF$ a foliation on $M$.
 Let $\CU$ be a path holonomy atlas for $\CF$.}
There is a canonical  {isomorphism of topological groupoids}
\[H(\pi^{-1}(\CF))\cong G(\pi^{-1}(\CU))
.\]
\end{prop}
\begin{proof}
{Lemma  \ref{lem:pullbackatlas}} shows that $\pi^{-1}(\CU)$ is an atlas of bisubmersions for $\pi^{-1}\CF$. Hence by Cor.  \ref{cor:adapted} 
 we   get the existence of a natural injective morphism {of topological groupoids}
{$\varphi\colon H(\pi^{-1}(\CF))\fto G(\pi^{-1}(\CU))$}. 

We show that $\varphi$ is surjective. As $H(\CF)$ is $\bs$-connected by lemma \ref{lem:sconnholgroupid}, and using that $\pi$ has connected fibres, we get that $\pi^{-1}(H(\CF))$ is an $\bs$-connected groupoid, and by proposition \ref{super} $G(\pi^{-1}(\CU))$ too. {It is a general fact that} any $\bs$-connected topological groupoid is generated by any symmetric neighbourhood of the identities, hence it suffices to show that the image of $\varphi$ contains a symmetric neighbourhood of the identities  $M$.

For this purpose, {recall that $\varphi$  is induced by morphisms of bisubmersions from a path holonomy atlas of $\pi^{-1}(\cF)$ to the atlas $\CU$ ({see the text after lemma \ref{lem:adapt}}).}
Notice that for any  path-holonomy bisubmersion $U$ of $\cF$, the pullback $\pi^{-1}(U)$ has the same dimension as a   path-holonomy bisubmersion   of $\pi^{-1}(\cF)$. Hence, by proposition \ref{path} ii), for any point $p\in \pi^{-1}(\bs(U)){\subset P}$ there is an isomorphism of bisubmersions from a neighborhood of $(0,p)$ in a path holonomy bisubmersion of ${\pi^{-1}(\cF)}$ to 
a neighbourhood of {$(p,(0,\pi(p)),p)$} in $\pi^{-1}(U)$. {Since }the quotient map $\pi^{-1}(U)\to G(\pi^{-1}(\CU))$ is an open map {by lemma \ref{lem:openmap}, we conclude that the image of $\varphi$ contains a neighbourhood of the identities}. 

{Finally, to show that the inverse map is continuous, use again lemma \ref{lem:adapt}.}
\end{proof}

 \begin{proof}[Proof of theorem. \ref{thm:pullbackgroid}]
{Let $\CU$ be a path holonomy atlas for $\CF$. 
We have a composition of isomorphisms
$$H(\pi^{-1}(\CF))\cong G(\pi^{-1}(\CU))\cong \pi^{-1}(H(\CF)),$$
where the first isomorphism is the one obtained in proposition \ref{prop:injmap} and the second isomorphism is given by proposition \ref{super} using $H(\CF)=G(\CU)$.}
 \end{proof}

\subsubsection{{Preservation of smoothness}}

{The isomorphism of theorem. \ref{thm:pullbackgroid} preserves smooth structures, whenever they are present. We now elaborate on this.} 

\begin{defi}\label{def:smoothgr}
Given a foliated manifold $(M,\CF)$ and  an atlas of bisubmersions $\CU=\{U_i : i\in I\}$ for $\CF$, we say that $G(\CU)$ is {\bf smooth} if there exists a ({necessarily unique})  smooth structure on it that makes the quotient map: $Q\colon \sqcup_i U_i \fto G(\CU)$ a submersion.
\end{defi}

It is easy to see that if $G(\CU)$ is smooth then it is a Lie {groupoid}.

\begin{lem}\label{lem:smoothtwo} 
\begin{enumerate}
 \item[(i)]  {Let $\CU$ be an atlas on a foliated manifold $(M,\CF)$, and let $\pi\colon P\fto M$ be a surjective submersion. Assume that $G(\CU)$ is smooth. Then   $G(\pi^{-1} \CU)$ is also smooth,  and the map $G(\pi^{-1} \CU)\cong\pi^{-1}(G(\CU))$ in proposition \ref{super} is an isomorphism of Lie groupoids}. 
 
 \item[(ii)]  Consider two atlases of bisubmersions $\CU_1$ and $\CU_2$, with $\CU_1$  {adapted to $\CU_2$}.  
Assume that the map {$\varphi \colon G(\CU_1)\fto G(\CU_2) $} from lemma \ref{lem:adapt} is surjective and
 $G(\CU_2)$ is smooth. Then   $G(\CU_1)$ is also smooth, {and $\varphi$ is an isomorphism of Lie groupoids}.
\end{enumerate}
\end{lem}
\begin{proof}
 (i) {
Being the pullback of a Lie groupoid by a submersion,  $\pi^{-1}(G(\CU))$ is a Lie groupoid.
Since the map in proposition \ref{super} is a homeomorphism, we can use it to transport the smooth structure on $\pi^{-1}(G(\CU))$ to $G(\pi^{-1}\CU)$.  
Since the quotient map $Q$ onto  $G(\CU)$ is a submersion, it follows that the map $Id\times Q\times Id$ given in eq. \eqref{map:pullbackatlas} is a submersion too. 
The commutativity of diagram \eqref{diag:iQi} implies that  $Q_{\pi}$ is a submersion  onto $G(\pi^{-1}\CU)$ endowed with the above smooth structure. The uniqueness in definition  \ref{def:smoothgr} finishes the argument.}

(ii)  {Since $\varphi$ is a homeomorphism by lemma \ref{lem:adapt} (ii), we can use it to transport the smooth structure on $G(\CU_2)$ to $G(\CU_1)$. 
We will show that the map $\widehat{\varphi}$ introduced just after lemma \ref{lem:adapt} is a submersion. Then the commutativity of diagram \eqref{diag:cont1} implies that  $Q_1$ is a submersion onto  $G(\CU_1)$ endowed with the above smooth structure. The uniqueness in definition  \ref{def:smoothgr} finishes the argument.
}

{
We now prove that $\widehat{\varphi}$ is a submersion. Let $u$ be a point in a bisubmersion $U \in \CU_1$. By definition  $\widehat{\varphi}(u)=Q_2(f(u))$,
  where $f$ is any  morphism of bisubmersions from a neighbourhood of $u$ to some bisubmersion  $V\in \CU_2$, and $Q_2 \colon V\to  G(\CU_2)$ is the projection map. There exists a morphism of bisubmersions $g\colon V'\to U_1$ (defined in a neighborhood $V'$ of $f(u)$) mapping $f(u)$ to $u$, by corollary \ref{carry} (ii). The following diagram commutes:}
   \begin{equation}
 \label{diag:contX}
\begin{tikzcd} U  \arrow[dr,swap,"\widehat{\varphi}"] & V' \arrow[l,swap, "g"] \arrow[d,"Q_2"]\\
  & G(\CU_2).
\end{tikzcd}
\end{equation}
{Indeed, for all $v\in V'$, we have $\widehat{\varphi}(g(v))=Q_2(f(g(v)))=Q_2(v)$, where the last equality holds because $f\circ g$ is a morphism of bisubmersions.
Since the derivative $d_{f(u)}Q_2$ is surjective by assumption, the commutativity of diagram \eqref{diag:contX} implies that the derivative $d_{u}\widehat{\varphi} \colon T_{u}U\to T_{\widehat{\varphi}(u)}G(\CU_2)$ is surjective. As $u$ was arbitrary, we conclude that $\widehat{\varphi}$ is a submersion.
 }
 \end{proof}

{The smooth version of theorem.  \ref{thm:pullbackgroid} is the following:}

\begin{prop}\label{prop:pullLie}
{Let $\pi\colon P\fto M$ be a surjective submersion with connected fibres and $\CF$ a foliation on $M$. If $H(\CF)$ is smooth then the map $\varphi\colon H(\pi^{-1}(\CF))
\cong\pi^{-1}(H(\CF))$ given in theorem. \ref{thm:pullbackgroid} {is an isomorphism of} \emph{Lie groupoids}.}
\end{prop}
\begin{proof}
{Let $\CU$ be a path holonomy atlas   of $\cF$. 
We check that the composition
$$H(\pi^{-1}(\CF))\cong G(\pi^{-1}(\CU))\cong \pi^{-1}(H(\CF))$$
appearing in the proof of theorem.  \ref{thm:pullbackgroid} is a composition of Lie groupoid isomorphisms.}

{
The second map is    a Lie groupoid isomorphisms, by lemma \ref{lem:smoothtwo} (i).}

{
The first map is    a Lie groupoid isomorphism: apply lemma \ref{lem:smoothtwo} (ii)
to  $\CU_2:=\pi^{-1}\CU$ (which 
an atlas of bisubmersion for $\pi^{-1}\CF$ by lemma  \ref{lem:pullbackatlas}), to a  path holonomy atlas  $\CU_1$  of $\pi^{-1}\cF$, and use that $\varphi$ is surjective (see  proposition \ref{prop:injmap}).}
\end{proof}

The holonomy groupoid of a foliated manifold $(M,\cF)$ is not always smooth, but by results of Claire Debord   \cite{Debord2013}, for any point $x\in M$ there is a smooth structure on the restriction of the holonomy {groupoid} to a leaf $L$ through $x$, {making it a Lie groupoid} (and  consequently on the isotropy group at $x$, {making it a Lie group}). More precisely, following  {\cite[Def. 2.8]{AZ1}, there exists a smooth structure on $H(\cF)_L$ -- the restriction of the holonomy groupoid  the leaf -- 
such that for any  path holonomy atlas $\{U_i\}$ for $\CF$,
the quotient map $Q_L\colon \sqcup_i (U_{i})_L \fto H(\cF)_L$ is a submersion.}

\begin{prop}\label{isotropull}
Let $\pi\colon P\fto M$ be a surjective submersion with connected fibres, $\CF$ be a foliation on $M$ and $\CU$ a path holonomy atlas for $\CF$.
{The map $\varphi\colon H(\pi^{-1}(\CF)) \cong\pi^{-1}(H(\CF))$ {of theorem  \ref{thm:pullbackgroid}} restricts to the following isomorphisms of Lie groupoids:}
\begin{enumerate}
\item[(i)] $(H(\pi^{-1}(\CF)))_{\widehat{L}}\cong (\pi^{-1}(H({\CF})))_{\widehat{L}}$, {for the restrictions to any leaf $\widehat{L}\subset P$,}
\item[(ii)] $(H(\pi^{-1}(\CF)))_p\cong (\pi^{-1}(H(\CF)))_p$, for the isotropy Lie groups at any $p\in P$.
\end{enumerate} 
\end{prop}

\begin{rem}
{There is canonical isomorphism of Lie groups} 
$(\pi^{-1}(H(\CF)))_p\cong H(\CF)_{\pi(p)}$.
\end{rem}
\begin{proof}
{We prove only (i), since (ii) is a direct consequence. 
Any leaf in $P$ is of the form $\widehat{L} =\pi^{-1}(L)$ for some leaf $L$ in $M$. We have $(\pi^{-1}(H(\cF)))_{\widehat{L}}= \widehat{L}{}_{\pi}\!\times_{\bt} (H(\cF)_{L}) {}_{\bs}\!\times_{\pi} \widehat{L}_{\pi}=
\pi^{-1}(H(F)_L)$}. Take a path holonomy atlas  $\CU=\{U_i\}_{i\in I}$ for $\CF$, note that $(\pi^{-1} U_i)_{\widehat{L}}=\widehat{L}{}_{\pi}\!\times_{\bt} ((U_i)_L) {}_{\bs}\!\times_{\pi} \widehat{L}_{\pi}$. The map $Q_L\colon \sqcup_i (U_{i})_L \fto H(\cF)_L$ is a submersion, by the above definition of smooth structure on $ H(\cF)_L$, therefore the map $$Id\times Q_L\times Id\colon \sqcup_i (\pi^{-1}U_{i})_{\widehat{L}} \fto (\pi^{-1}H(\cF))_{\widehat{L}}$$ is a submersion.

{This allows us to apply the arguments of the proof of lemma \ref{lem:smoothtwo}   to groupoids over $\widehat{L}$ (rather than over $P$). The proof of proposition \ref{prop:pullLie} delivers the desired conclusion.}
\end{proof}

\subsection{Morita equivalence for open {topological} groupoids}\label{subsec:hol3}

{The holonomy groupoid of a singular foliation (see definition \ref{def:holgroidph}) is not a Lie groupoid in general, but just an open topological groupoid. {After discussing Morita equivalence for open topological groupoids, we state the main result of the paper, theorem \ref{thm:MEfolgroids}.}

\begin{defi}
A topological groupoid is {\bf open} if its source and target maps are open maps.
\end{defi}

\begin{prop} 
The holonomy groupoid of a foliated manifold $(M,\CF)$ is an open topological groupoid.
\end{prop}
\begin{proof}
Take a path holonomy atlas $\CU$ and denote by $Q\colon \sqcup_{U\in \CU} U\fto H(\CF)$ the quotient map. 
The following diagram commutes, {where we denote by $\bs_H$ the source map of the holonomy groupoid}:
\[\begin{tikzcd}
\sqcup_{U\in \CU} U\arrow[dr," {\bs}"] \arrow[d,swap,"Q"] & \\
H(\CF) \arrow [r,swap,"\bs_H"] & M 
\end{tikzcd}\]
{Recall from \S \ref{sec:holconstr} that $H(\CF)$ is endowed with the quotient topology}.
Using that $Q$ is continuous {and surjective}, and that $ {\bs}$ is a submersion and therefore   an open map, it  follows that $\bs_H$ is open map. A similar argument can be used for $\bt_H$.
\end{proof}

{The proof of the following lemma can be found in \cite[lemma 2.25]{TuNonHaus}.} 
\begin{lem}\label{lem:openmapABC} Let $A,B,C$ be topological spaces and $f\colon A\fto C$ {a continuous map}. If $g\colon B\fto C$ is a continuous and \emph{open} map, then {$\text{Pr}_1\colon A {}_f\!\times_g B\fto A$} is also a {continuous and \emph{open} map}, where the domain is endowed with the subspace topology.
	
{Moreover, if {$g$} is surjective then {$\text{Pr}_1$} also is}.
\[\begin{tikzcd}
A {}_f\!\times_g B \arrow[r,"\Pr_2"] \arrow[d,swap, "\Pr_1"]& B \arrow[d, "g"]  \\
  A \arrow[r, "f"] & C\\
\end{tikzcd}\]

\end{lem}

\begin{prop} 
If $G\soutar M$ is an open topological groupoid and $\pi\colon P\fto M$ is a continuous open and surjective map, then $\pi^{-1} G$ is an open topological groupoid.
\end{prop}
\begin{proof}
 {We show that the target map of $\pi^{-1} G$ is open. 
The first projection $\text{Pr}_1$ of
 $G {}_\bs\!\times_\pi P$ is open by lemma \ref{lem:openmapABC}, since $\pi\colon P \to M$ is open (being a submersion). Hence the composition $\bt\circ\text{Pr}_1$ is  open. Again by lemma \ref{lem:openmapABC}, this implies that the first projection of 
 $P {}_\pi \!\times_{\bt\circ\text{Pr}_1}\left(G {}_\bs\!\times_\pi P\right)= \pi^{-1} G$ is  open, and this is precisely the target map of $\pi^{-1} G$.} For the source map, proceed similarly.
\end{proof}

{Morita equivalence of topological groupoids can be defined in terms of weak equivalences as in \cite[\S 1.3]{DPronk}. We argue in the appendix (see proposition \ref{prop:AllME} and remark \ref{rem:topopenME}) that for open topological groupoids, Morita equivalence can be alternatively be stated as follows.}  

\begin{defi}\label{def:MEtopgroid}
Two open topological groupoids $G\rightrightarrows M$ and $H\rightrightarrows N$ are \textbf{Morita equivalent} if   there exists a topological space $P$, and two surjective open maps $\pi_M\colon P\fto M$ and $\pi_N\colon P\fto N$ such that $\pi^{-1}_M G\cong \pi_N^{-1}H$ as topological groupoids.

In this case we call $(P,\pi_M,\pi_N)$ a {\bf Morita equivalence} between $G$ and $H$.
\end{defi}

{We can finally state the main result of the paper:}

 \begin{thm}\label{thm:MEfolgroids}
Hausdorff Morita equivalent singular foliations have holonomy groupoids which are Morita equivalent as open topological groupoids.
\end{thm} 
\begin{proof}
{Apply twice theorem \ref{thm:pullbackgroid}, noticing that submersions are open maps.}
\end{proof}

{Applying proposition \ref{prop:pullLie}
we can specialize the above result to projective (for example, regular) foliations:
 if two projective foliations are Hausdorff Morita equivalent, then their holonomy groupoids  are Morita equivalent as Lie groupoids. Combining   with  proposition 
\ref{prop:MEholtoMEfol} we obtain:

 \begin{prop}\label{prop:equivproj}
Provided  their holonomy groupoids are Hausdorff, two projective  foliations  are  
Hausdorff Morita equivalent  if{f} their holonomy groupoids   are Morita equivalent as  Lie groupoids. 
 \end{prop}
}

 \subsection{Holonomy transformations}\label{subsec:holtrafo}

{Given a regular foliation, a classical construction associates to every path in a leaf its holonomy (a germ of diffeomorphism between slices transverse to the foliation). We review the extension of this construction to singular foliations
\cite[\S 2]{AZ2} and show that it is invariant under   pullbacks.}

\begin{defi}\label{holfrom}
Let $(M,\cF)$ be a singular foliation, and $x,y\in M$ lying in the same leaf.  
Fix a transversal $S_{x}$ at $x$, as well as a transversal $S_y$ at $y$. A \emph{holonomy transformation  from $x$ to $y$} is an element of 
$$ \frac{GermAut_{\cF}(S_x, S_y)}{exp(I_x \cF)|_{S_x}}.$$ 
Here $GermAut_{\cF}(S_x,S_y)$ is the space of germs at $x$ of locally defined diffeomorphisms preserving $\cF$   mapping $S_x$ to $S_y$, restricted to $S_x$.
Further $exp(I_x \cF)|_{S_x}$
is the space of germs at $x$ of time-one flows of time-dependent vector fields in $I_x \cF$ mapping $S_x$ to itself, restricted to $S_x$.
\end{defi}

Holonomy transformations are relevant because the holonomy groupoid maps canonically into them \cite[theorem. 2.7]{AZ2}.

\begin{thm}\label{globalaction} Let $x, y \in (M,\cF)$ be points in the same leaf $L$, and fix  transversals $S_x$ at $x$ and $S_y$ at $y$. 
Then there is a well defined map 
\begin{align}\label{Phixy}
\Phi^{\cF} \colon H(\cF)_x^y \rightarrow \frac{GermAut_{\cF}(S_x, S_y)}{exp(I_x \cF)|_{S_x}},\quad h \mapsto \langle \tau\rangle.   
\end{align}

Here $\tau$ is defined as follows, given $h \in H(\cF)_x^y{:=\bt^{-1}(y)\cap \bs^{-1}(x)}$:  
\begin{itemize}
\item take any bisubmersion $(U,\bt,\bs)$ in the path-holonomy atlas with a point $u\in U$ satisfying $[u]=h$,
\item  take any section $\bar{b} \colon S_x \to U$  of $\bs$ through $u$ {transverse to the $\bt$-fibers} such that $(\bt\circ \bar{b})(S_x)\subset S_y$, 
\end{itemize}
and define $\tau:=\bt\circ \bar{b} \colon S_x \to S_y$. 
\end{thm}
For all $x,y$ the map  $\Phi^{\cF}$ is injective \cite[theorem. 2.20]{AZ2} {and assembles to a groupoid morphism \cite[theorem. 2.7]{AZ2}}.
In the case of regular foliations, the map $\Phi^{\cF}$ describes the usual geometric notion of holonomy.

\begin{rem}\label{rem:linholtrafo}
{Linearizing any representative of $\Phi^{\cF}(h)$ one associates to $h$ a well-defined linear map $T_xS_x\to T_yS_y$. Notice that $T_xS_x$ can be identified with the normal space $N_xL$ to the leaf at $x$. Hence, when $x=y$, we obtain a representation of the isotropy Lie group $H(\cF_M)_x^x$ on $N_xL$ \cite[\S 3.1]{AZ2}.}
\end{rem}

Let $(M,\CF)$ be a foliated manifold and $\pi\colon P\fto M$ a surjective submersion with connected fibers. Recall that there is  
a canonical surjective morphism $$\Pi\colon H(\pi^{-1}(\CF)) \cong\pi^{-1}(H(\CF)) \to H(\CF),$$ {where the isomorphism is given in theorem. \ref{thm:pullbackgroid}.}
We now show that the holonomy transformations associated to a point in $H(\pi^{-1}(\CF)) $ and to its image coincide. 
 
\begin{prop}\label{thm:holtrafo}

For every  $h\in H(\pi^{-1}(\CF))$, the holonomy transformation associated to $h$ and    to $\Pi(h)$ coincide, under the obvious identifications. More precisely: fix  slices  $S_{x}$ at $x:=\bs(h)\in P$ and $S_y$ at $y:=\bt(h)$, transversal to  $\pi^{-1}(\cF)$. Then
$$\Phi^{\pi^{-1}(\cF)}(h) \in \frac{GermAut_{\pi^{-1}(\cF)}(S_x, S_y)}{exp(I_x \pi^{-1}(\cF))|_{S_x}}$$
and
$$\Phi^{\cF}(\Pi(h)) \in \frac{GermAut_{\cF}(S_{\pi(x)}, S_{\pi(y)})}{exp(I_{\pi(x)} \cF)|_{S_{\pi(x)}}}$$
coincide under the diffeomorphisms $S_x\cong S_{\pi(x)}:=\pi(S_x)$ and $S_y\cong S_{\pi(y)}:=\pi(S_y)$ obtained restricting $\pi$.
\end{prop} 
\begin{proof}
Let  $h\in H(\pi^{-1}(\CF))$. By theorem. \ref{globalaction}, 
${\Phi^{\pi^{-1}(\cF)}(h)}$
is obtained using  a bisubmersion $V$ in the path-holonomy atlas of $(P,\pi^{-1}(\cF))$, a point $v\in V$ with $[v]=h$, and 
a certain section through $v$. {By  proposition \ref{prop:injmap}}  
the groupoid $ H(\pi^{-1}(\CF))$ {is isomorphic to  $G(\pi^{-1}(\CU))$, which} is constructed out of the atlas $\pi^{-1}(\CU)$ where $\CU$ is a path-holonomy atlas for $(M,\cF)$. 
{This means that} there is a bisubmersion $U$ in $\CU$ and a morphism of bisubmersions $$\psi\colon V\to \pi^{-1}(U)$$ defined near $v$. We have $\psi(v)=(y,u,x)\in \pi^{-1}(U)$ for some $u\in U$. Further, applying $\psi$ to any bisection of $V$ we obtain a bisection of 
$\pi^{-1}(U)$ carrying the same diffeomorphism. Hence we can work on the latter bisubmersion  instead of on $V$.

Take any section $\bar{b} \colon S_x \to \pi^{-1}(U)$ of $\bs$ through $(y,u,x)$  {transverse to the $\bt$-fibres} such that $(\bt\circ \bar{b})(S_x)\subset S_y$.
{Due to the diffeomorphism $S_x\cong S_{\pi(x)}$}, there is a unique section $b\colon S_{\pi(x)}\to U$ through $u$ such that $$\bar{b}(p)=(*,b(\pi(p)),p)$$ for any $p\in S_x$. (Here $*$ denotes the unique point of $S_y$ that corresponds to $(\bt\circ b)(\pi(p))$ under the identification $S_y \cong  S_{\pi(y)}$.)
The diffeomorphisms $$\bt\circ \bar{b}\colon S_x \to S_y \;\;\text{      and      }\;\;
\bt\circ  {b}\colon S_{\pi(x)} \to S_{\pi(y)}$$ coincide under the natural identification between slices. The former is a representative of $\Phi^{\pi^{-1}(\cF)}(h)$, while the latter is a representative  of $\Phi^{\cF}([u])$. {We conclude noticing that $[u]=\Pi(h)$, as can seen using the proof of} proposition \ref{super}.
\end{proof}

 \subsection{Further invariants}
\label{subsec:furtherinv}

The results of {\S \ref{sec:hol2} and \S 
\ref{subsec:holtrafo}} 
allow  to find more refined invariants than those listed in \S\ref{subsec:firstinv}. 
 
 \begin{thm}\label{thm:furtherinv}
Let $(M,\cF_M)$ and $(N,\cF_N)$ be {Hausdorff} Morita equivalent singular foliations. 
Fix a Hausdorff Morita equivalence, and let $x\in M$ and $y\in N$ be points in corresponding leaves (see proposition \ref{prop:corrleaves}). Then

\begin{itemize}
\item[(i)] The isotropy Lie groups $H(\cF_M)_x$ and $H(\cF_N)_y$ are isomorphic as Lie groups. 
\item[(ii)] {If $h_x\in H(\cF_M)_x$ and $h_y \in H(\cF_M)_y$ correspond under the Lie group isomorphism in (i), then their holonomy transformations  $\Phi^{\cF_M}(h_x)$ and $\Phi^{\cF_N}(h_y)$ coincide under the natural identification between slices. 
}
\item[(iii)] The representations of corresponding isotropy Lie groups on   
normal spaces to the leaves are isomorphic.
\end{itemize}
\end{thm}
\begin{proof}
(i) follows from proposition \ref{isotropull}.

(ii) follows from proposition \ref{thm:holtrafo}.

(iii) is a consequence of (ii), {in view of Remark \ref{rem:linholtrafo}}.
\end{proof}

\subsection{A second look at  {Hausdorff} Morita equivalence of singular foliations}\label{sec:secondlook}

In the definition of {Hausdorff} Morita equivalence between two singular foliations  $(M,\cF_M)$ and $(N,\cF_N)$, Def. \ref{def:defMEfol}, it is required that the maps
$\pi_M\colon P\fto M$ and $\pi_N\colon P\fto N$ be \emph{surjective submersions  with connected fibers}.
It is tempting to think that {Hausdorff} Morita equivalence of  singular foliations
can be phrased weakening these three conditions, i.e. that adopting weaker conditions one obtains the same equivalence classes of singular foliations. This is not the case:

\begin{prop}\label{prop:3conditions}
We do not obtain the same equivalence classes of singular foliations if we replace any of the three conditions in Def. \ref{def:defMEfol} as follows:
\begin{itemize}
\item ``Surjective'' by ``meets every leaf of the singular foliation'',
\item ``Submersion'' by ``is transverse to the singular foliation''  {\cite[Def. 19]{AndrSk}},
\item ``{With} connected fibres`` by ``such that the preimages of leaves are connected''.
\end{itemize}
\end{prop}

\begin{rem}
{The first two items above are motivated by what occurs for Lie groupoids.}
The Morita equivalence of two Lie groupoids $G\rightrightarrows M$ and $H\rightrightarrows N$  can be equivalently phrased by replacing the condition 
that the maps $\pi_M\colon P\fto M$ and $\pi_N\colon P\fto N$ in Def. \ref{def:MEgroid} are   surjective submersions with the following condition: these maps are transversal to the orbits of the Lie groupoids  ($G$ and  $H$ respectively) and
meet every orbit. This fact can be found in \cite{MackenzieGrdAld} and \cite{JoaoCrainic2017}, {and  follows also from proposition \ref{prop:AllME}.}

\end{rem}

{To prove proposition \ref{prop:3conditions}} it suffices to display examples of maps $\pi\colon P\fto (M,\CF)$ in which {each} of the conditions on the left hand side is weakened and so that the holonomy groupoid $H(\pi^{-1}\CF)$ is \emph{not} Morita equivalent to $H(\CF)$. Indeed, in this case, $(P,\pi^{-1}\CF)$ and  $(M,\CF)$ can not be {Hausdorff} Morita equivalent, due to  {theorem. \ref{thm:MEfolgroids}}.

We now display the examples mentioned above, involving only regular foliations. 

\begin{ex}\label{ex:surj}({\bf ``Surjective'' is needed}) Take $M$ to be the Moebius band  $$M:= \RR\times (-1,1)/\sim$$ where $(x,y)\sim (x+3k,(-1)^k y)$ for $k\in \ZZ$. Take $$P:= M\backslash \{\overline{(2,0)}\},$$ the Moebius band without a point in the ``middle circle" (the equivalence class of $(2,0)$).

\noindent Let  $$\pi:P\hookrightarrow M$$ be the inclusion. On $M$ take the regular (rank one) foliation $\CF$ given by horizontal vector fields, then $\pi^{-1}(\CF)$ is also given by horizontal vector fields. Note that  $\pi$ is a submersion with connected fibres that meets every orbit, but it is not surjective.

The isotropy group at $\overline{(0,0)}\in M$ of the holonomy groupoid $H(\CF)$ is  
isomorphic to $\ZZ_2$. But the isotropy group at $\overline{(0,0)}\in P$ of the holonomy groupoid $H(\pi^{-1}\CF)$ is trivial (the leaf through that point is contractible). Therefore the two holonomy groupoids can not be Morita equivalent.

\begin{figure}[h]
	 \caption{The {manifold} $P$}
	\centering
	\scalebox{0.5}{\includegraphics{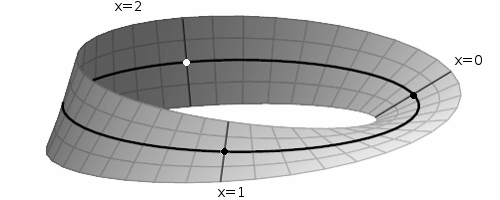}}
\end{figure}
\end{ex}

\begin{ex}({\bf ``Submersion'' is needed}) Take  $P:=\RR \sqcup \left(\RR\backslash\{0\}\right)$, $M:=\RR$, and define $\pi\colon P\fto M$ {so that it sends the copy of $\RR$ to the point $0\in M$ and $\RR\backslash\{0\}$ to $M$ by the inclusion}. On $M$ take the full foliation. The map $\pi$ is surjective, has connected fibres and it is transverse to the foliation in $M$, but it is not a submersion. 

The pullback foliation on $P$ is also the full foliation, but $P$ has three connected components. Hence the spaces of leaves are not homeomorphic and the holonomy groupoids  are not Morita equivalent. 
\end{ex}

\begin{ex}({\bf ``{With} connected fibres'' is needed}) 
This example is a variation of Ex. \ref{ex:surj}.
 Take the Moebius band $M$  as in that example. {Let $$M':= \RR\times (-1,1)/\sim'$$ where $(x,y)\sim' (x+k,(-1)^k y)$ for $k\in \ZZ$. Notice that $M'$ is a smaller Moebius band, and since the equivalence classes of $\sim$ are contained in those of $\sim'$, there is a 
 natural quotient map  $q\colon M\to M'$ which is a $3$ to $1$ covering map.} 
 
Let $P$ be $M$ with a point removed,  as in  Ex. \ref{ex:surj}. {Let} $$\pi '\colon P\fto M'$$ be  {the restriction of $q$ to $P$}. On $M'$ take the regular (rank one) foliation given by horizontal vector fields. Then $\pi '$ is a surjective submersion with connected preimages of leaves, but whose fibres are not connected (all fibres consist of three points, except for one that consists of two points). As in the first example,  the isotropy groups of the corresponding holonomy groupoids are $\ZZ_2$ at the point $\overline{(0,0)}\in M '$ and  the trivial group at $\overline{(0,0)}\in P$. Hence the  holonomy groupoids can not be Morita equivalent.
\end{ex}

 %%%%%
 %%%%%%%%
 
 \section{Further developments}\label{sec:NHMESingfol}
\subsection{{An extended   equivalence for singular foliations}}

 {Our notion of  Hausdorff 
Morita equivalence (Def. \ref{def:defMEfol}) has certain drawbacks, which originate from the fact that {the}  space of arrows of a Lie groupoid is not necessarily Hausdorff:}

\begin{itemize}
\item If two non-Hausdorff Lie groupoids are Morita equivalent, then their singular foliations might  {not} be Hausdorff Morita equivalent. (Compare with proposition \ref
{prop: implications}).
\item As a consequence, we have to add a Hausdorfness assumption\footnote{Notice  that a  Hausdorfness assumption is needed also to match the Morita equivalence of Lie groupoids and Lie algebroids, see proposition \ref{prop:MEalggroids}.} in proposition \ref{prop:equivproj} on projective foliations.
\end{itemize}  
 
In an attempt to extend the notion of Hausdorff Morita equivalence so that the above drawbacks do not occur, we propose to allow the manifold $P$ in Def. \ref{def:defMEfol} to be \emph{non-Hausdorff}.

A first issue to address is the notion of singular foliation on a  non-Hausdorff manifold.  
 In remark \ref{rem:sheavesbij} we saw that on a (Hausdorff) manifold,
 Def. \ref{def:singfol}  (in terms of compactly supported vector fields) is    equivalent to the characterisation given in that remark (in terms of subsheaves). 
On a non-Hausdorff manifold   $V$, this is no longer the case. Indeed the notion obtained extending trivially Def. \ref{def:singfol} is quite restrictive, 
the main reason being that there might be points $p\in V$ where all compactly supported  vector fields vanish. However the sheaf of smooth vector fields on $V$ (a sheaf of $C^{\infty}$-modules) is well-behaved.
Hence we propose to define a  singular foliation on a possibly non-Hausdorff manifold $V$ as \emph{an involutive, locally finitely generated subsheaf of the sheaf of smooth vector fields.} 

A second issue to address is how to
 extend the notion of pullback foliation to a non-Hausdorff manifold. By  remark \ref{rem:sheavesbij}, for a Hausdorff manifold the sheaf associated to a pullback foliation is given by 
 $\widehat{\iota_U^{-1}(\pi^{-1}\CF)}=\widehat{\pi|_U^{-1}\CF}$ for every open subset $U$.
For a non-Hausdorff manifold $V$ and a submersion $\pi:V\fto M$ to a manifold, we define the pullback foliation as the following 
{subsheaf} $\cS^{\pi^{-1}\CF}$ of the sheaf of vector fields: for any open (possibly non-Hausdorff) subset $U\subset V$,
\[\cS^{\pi^{-1}\CF}(U):=\{X\in\CX(U) \st X|_H\in \widehat{\pi|_H^{-1} \CF} \text{ for all open Hausdorff subsets }H\subset U\}\]

With the above ingredients at hand we can propose the following definition.

\begin{defi}\label{def:NHMEfol}
	Two singular foliations  $(M,\cF_M)$ and $(N,\cF_N)$
	are {\bf Morita equivalent} if there exists a \emph{possibly non-Hausdorff} manifold $P$ and two \emph{surjective submersions with  connected Hausdorff fibres} $\pi_M\colon P\fto M$ and $\pi_N\colon P\fto N$  such that $\cS^{\pi_M^{-1}\CF_M}=\cS^{\pi_N^{-1}\CF_N}$ as subsheaves of $\CX_P$.
	\[\begin{tikzcd}
	& P \arrow[dl,swap, "\pi_M"] \arrow[dr, "\pi_N"]   &\\
	(M,\CF_M)& &(N,\CF_N)   
	\end{tikzcd}\]
\end{defi}

We then expect 
\begin{itemize}
\item  the following extension of proposition \ref{prop: implications}: \emph{if two ({possibly} non-Hausdorff) Lie groupoids are Morita equivalent, then their singular foliations are Morita equivalent}.
\item  to carry out the construction of the holonomy groupoid (Def. \ref{def:holgroidph}) starting from the sheaf-theoretic characterization of 
  singular foliation, even for a non-Hausdorff foliated manifold. Further we expect the following improvement of theorem. \ref{thm:MEfolgroids} to hold:  

\emph{Morita equivalent singular foliations have holonomy groupoids which are Morita equivalent as open topological groupoids.} 
\item  the following improvement of proposition \ref{prop:equivproj}:
\emph{
Two projective singular foliations are Morita equivalent if{f} their holonomy groupoids  are Morita equivalent as  Lie groupoids.}
\end{itemize} 

%%%%%%%%%%%%%%%%%%%%%%
%%%%%%%%%%%%%%%%%%%%%%
\appendix

\section{Appendix}
 
\subsection{Morita equivalence for open topological groupoids and Lie groupoids}\label{sec:MEGrpd}

{Definition \ref{def:MEgroid}, on Morita equivalence of Lie groupoids,   is equivalent to several other characterisations, as   was proved in \cite{MdMkGrd},   {(see also \cite{LauMatXu}).} An analogue statement holds also for open topological groupoids, upon replacing   submersions with continuous open  maps. In this appendix we recall these facts and prove some implications that are used in the main body of the paper, {the main one being corollary \ref{cor:SConnMEGrpd2}.}

{We start recalling the notion of weak equivalence, {as given in \cite[\S 1.3]{DPronk}}, and of  bitorsor.}

\begin{defi}\label{def:weakeq}
Let $G\soutar M$ and $\CG\soutar P$ be two Lie groupoids (respectively, topological groupoids). A morphism $\widehat{\pi}\colon\CG\fto G$ is a {\bf weak equivalence} if:
\begin{enumerate}
\item[(i)] $\CG\fto \pi^{-1} G; \;\;\gamma\mapsto (\bt(\gamma),\widehat{\pi}(\gamma),\bs(\gamma))$ is an isomorphism,
\item[(ii)] $\bt\circ\text{Pr}_1\colon G {}_{\bs}\!\times_{\pi} P\fto M$ is a surjective submersion (resp. a surjective continuous and open map).
\end{enumerate}
{Here $\pi\colon P\to M$ denotes the base map covered by $ \widehat{\pi}$.}
\end{defi}

\begin{rem}\label{rem:weakeq}
\begin{itemize}
\item[i)] Looking at a groupoid as a small category, a weak equivalence is the same thing as a {fully}  faithful and essentially surjective functor.
\item[ii)] When a map $\pi\colon P\fto M$ is completely transverse (transverse to the orbits and meeting every orbit) to a Lie groupoid $G\soutar M$, then the projection $\pi^{-1} G \fto G$ is a weak equivalence.
\item[iii)] {If $G$ is an open topological groupoid and ${\pi}$ is a continuous, open and surjective map, then condition (ii) in definition \ref{def:weakeq} is automatically satisfied, as can be showed using}
 lemma \ref{lem:openmapABC}.
 \end{itemize}
\end{rem}

\begin{defi}
Let $G\soutar M$ be a Lie groupoid (respectively, a topological open groupoid) and $\pi\colon P\fto M$ a surjective submersion (resp. a surjective continuous and open map). A {\bf $G$-action} over a \emph{not necessarily Hausdorff manifold} $P$ is a smooth (resp. continuous) map $\star \colon G {}_\bs\!\times_\pi P \fto P$ such that for all $g,h\in G$ and $p\in P$:
$$ \pi(g\star p)=\bt(g),\;\;\;\;\;\;\;\;
 g\star (h \star p)=(gh)\star p,\;\;\;\;\;\;\;\;
 e_{\pi(p)}\star p=p.$$
Such a manifold $P$ with a $G$-action is called a {\bf $G$-module} and $\pi$ is called its {moment map}. 
If the $G$-action is free and proper 
{then} 
$P/G$ is a manifold {and} we say that $P$ is a {\bf $G$-principal bundle}.

A {\bf $(G,H)$-bimodule} for the Lie groupoids $G\soutar M$ and $H\soutar N$ is a (not necessarily Hausdorff) manifold $P$ with two actions commuting with each other. 
A $(G,H)$-bimodule $P$ that is principal with respect to both actions and such that $G\backslash P\cong N$ and $P/H\cong M$  is called a {\bf $(G,H)$-bitorsor}.
\end{defi}
The following statement can be found in \cite[\S 2.5]{MdMkGrd}
\begin{lem}\label{lem:bitorsor}
Consider a {Lie} groupoid $\Gamma\soutar K$, a $\Gamma$-principal bundle $S$, a $\Gamma$-module $Q$ and a map $f\colon Q\fto S$ preserving the $\Gamma$ actions. Then $Q/\Gamma$ is a manifold.
\end{lem}
 
\begin{prop}\label{prop:AllME} Let $G\soutar M$ and $H\soutar N$ be Lie groupoids. The following statements are equivalent: 
\begin{enumerate}
\item[(i)] There exists a Lie groupoid $\Gamma$ and two weak equivalences $\CG\fto G$ and $\CG\fto H$.
\item[(ii)] There exists a $(G,H)$-bitorsor $P$.
\item[(iii)] $G$ and $H$ are Morita equivalent ({Def. \ref{def:MEgroid}}).
\end{enumerate}
\end{prop}

The proof of this statement can be found in \cite{MdMkGrd} and \cite[prop. 2.4]{LauMatXu}, nevertheless we review its proof here.  

\begin{proof}
$(i)\Rightarrow (ii)$: Consider a Lie groupoid $\CG\soutar K$ with weak equivalences $\widehat{\pi}_M\colon\CG\fto G$ and $\widehat{\pi}_N\colon\CG\fto H$. Therefore $\Gamma=\pi_M^{-1} G\cong \pi_N^{-1} H$. We get that $Q_G:=G {}_{\bs}\!\times_{\pi_M} K$ is a $(G,\Gamma)$-bitorsor. Using a similar argument we get that $Q_H$ a $(\Gamma,H)$-bitorsor.   The (not necessarily Hausdorff) manifold $Q:=(Q_G\times_K Q_H)$ 
{has a diagonal $\Gamma$-action with the canonical map to  $K$ {as moment map}. Applying lemma \ref{lem:bitorsor} to the map $Q\fto Q_G$ we see that
 $$P:=(Q_G\times_K Q_H)/\Gamma$$ is a (not necessarily Hausdorff) manifold.} One can  check that it is a $(G,H)$-bitorsor. 

$(ii)\Rightarrow (iii)$:  Consider a $(G,H)$-bitorsor $P$, with moment maps $\pi_M\colon P\fto M$,$\pi_N\colon P\fto N$. 
{Then $$\Gamma= G {}_\bs\!\times_{\pi_M} P {}_{\pi_N}\!\times_\bt H$$  has a} natural structure of Lie groupoid over $P$ with $\bt(g,p,h)=gph$, $\bs(g,p,h)=p$ and multiplication given canonically by $G$ and $H$. Then the maps $\Gamma\fto \pi_M^{-1}G; \;(g,p,h)\mapsto (p,g^{-1},gph)$ and $\Gamma\fto \pi_N^{-1} H;\;(g,p,h)\mapsto (p,h,gph)$ are isomorphisms of Lie groupoids.

 This shows that $\pi_M^{-1}G\cong \pi_N^{-1} H$ as Lie groupoids over {the} not necessarily Hausdorff manifold $P$. Now take  a Hausdorff cover $\{U_i\}_{i\in I}$ of $P$ and let $\tilde{P}:=\sqcup_i U_i$. There is a canonical submersion $\pi\colon \tilde{P}\fto P$. It is easy tho see that $(\tilde{P},\pi_M\circ \pi, \pi_N\circ \pi)$ is a Morita equivalence. 
 
$(iii)\Rightarrow (i)$: Given  a Morita equivalence $(P,\pi_M,\pi_N)$ between $G$ and $H$, call $\Gamma:=\pi_M^{-1}G\cong \pi_N^{-1}H$. The {natural projections $\Gamma\fto G$ and $\Gamma\fto H$ are weak equivalences}.
\end{proof}

 \begin{rem}\label{rem:topopenME}
{proposition \ref{prop:AllME} also holds for open topological groupoids,
as can be proven using lemma \ref{lem:openmapABC}}. {For arbitrary topological groupoids, this is not the case.}
\end{rem}

\begin{cor}\label{cor:SConnMEGrpd2}
{Let $k\ge0$.} If $G\soutar M$ and $H\soutar N$ are source $k$-connected Morita equivalent {Hausdorff} Lie groupoids, then {there exists a Hausdorff $(G,H)$-bitorsor $P$}.
Moreover this bitorsor is a Morita equivalence with $k$-connected fibres ({in the sense of Def. \ref{def:MEgroid})}.
\end{cor}

\begin{proof}   Following the implications $(iii)\Rightarrow (i)\Rightarrow (ii)$ in the proof of proposition \ref{prop:AllME}, one sees that the bitorsor $P$ constructed there is Hausdorff. {Then use the implication $(ii)\Rightarrow(iii)$ to prove that $P$ is a Morita equivalence.}
	
{Note that} being $P$ a bitorsor, the fibres of $\pi_M\colon P\fto P/H{\cong M}$ are {diffeomorphic} to the source fibers of $H$, {which are $k$-connected by assumption}. A similar argument holds for the fibres of $\pi_N\colon P\fto P/G\cong N$.
\end{proof}

\begin{rem}
The Morita equivalence $P$   in corollary \ref{cor:SConnMEGrpd2} is a (global) bisubmersion for the underlying foliations, as we now show.
Using the implication ``$(ii)\Rightarrow (iii)$'' in proposition \ref{prop:AllME} we get an isomorphism of {Lie} groupoids $$\pi_M^{-1}G\cong \pi_N^{-1} H\cong G\ltimes P\rtimes H.$$ 
Denote by $\CF_M$ and $\CF_N$ the  foliations underlying $G$ and $H$. Using {lemma \ref{GtoA} i)} and lemma \ref{lem:pullbackalgfol} we get that the foliations  underlying $\pi_M^{-1}G$ and $\pi_N^{-1} H$ are $\pi_M^{-1}\CF_M$ and $\pi_N^{-1}\CF_N$ respectively.
Since $P$ is a {$(G,H)$-bitorsor,} the foliation underlying  the Lie groupoid $G\ltimes P\rtimes H$ is $\CG_c(ker(d\pi_M))+\CG_c(ker(d\pi_N))$. Hence $$\pi_M^{-1}\CF_M=\pi_N^{-1}\CF_N=\CG_c(ker(d\pi_M))+\CG_c(ker(d\pi_N)).$$
\end{rem}

\bibliographystyle{habbrv}
%\bibliography{MEfolbib}

\end{document}